\documentclass[11pt]{article}

\usepackage{amsmath}
\usepackage{epsfig, amssymb, amsfonts, dsfont}
\usepackage{amsthm}
\usepackage{amstext}
\usepackage{amsopn}
\usepackage{mathrsfs}
\usepackage{subfigure}
\usepackage{graphicx}
\usepackage[left=2.3cm,top=2cm,right=2.3cm,bottom=2cm, nohead,foot=1cm]{geometry}
\usepackage[makeroom]{cancel}
\usepackage{color}
\usepackage{enumerate}
\usepackage{pdfrender,xcolor}
\usepackage{comment}
\usepackage{stackrel}
\usepackage{mathtools}
\usepackage{amsbsy}
\usepackage{tikz}
  \usetikzlibrary{automata,topaths}

\makeatletter
\newcommand*\rel@kern[1]{\kern#1\dimexpr\macc@kerna}
\newcommand*\widebar[1]{%
  \begingroup
  \def\mathaccent##1##2{%
    \rel@kern{0.8}%
    \overline{\rel@kern{-0.8}\macc@nucleus\rel@kern{0.2}}%
    \rel@kern{-0.2}%
  }%
  \macc@depth\@ne
  \let\math@bgroup\@empty \let\math@egroup\macc@set@skewchar
  \mathsurround\z@ \frozen@everymath{\mathgroup\macc@group\relax}%
  \macc@set@skewchar\relax
  \let\mathaccentV\macc@nested@a
  \macc@nested@a\relax111{#1}%
  \endgroup
}
\makeatother

\numberwithin{equation}{section}

\newtheorem{theorem}{Theorem}[section]

\newtheorem{lemma}[theorem]{Lemma}

\newtheorem{proposition}[theorem]{Proposition}

\theoremstyle{definition}
\newtheorem{definition}[theorem]{Definition}
\newtheorem{remark}[theorem]{Remark}

\newcommand{\R}{{\mathbb R}}
\newcommand{\E}{{\mathbb E}}

\DeclareFontFamily{U}{mathx}{\hyphenchar\font45}
\DeclareFontShape{U}{mathx}{m}{n}{
      <5> <6> <7> <8> <9> <10>
      <10.95> <12> <14.4> <17.28> <20.74> <24.88>
      mathx10
      }{}
\DeclareSymbolFont{mathx}{U}{mathx}{m}{n}
\DeclareMathSymbol{\bigtimes}{1}{mathx}{"91}

\usepackage{relsize}

\newcommand{\Z}{{\mathbb Z}}
\newcommand{\N}{{\mathbb N}}

\newcommand{\sbf}{{s}}
\newcommand{\bbf}{{b}}

\date{\vspace{-9ex}}

\title{Weak-disorder limit for directed polymers on critical \\ hierarchical graphs  with vertex disorder}

  \author{ \textbf{Jeremy Clark}\footnote{ {\tt
jeremy@olemiss.edu}}  \,\, \text{ and} \,\, \textbf{Casey Lochridge}\footnote{ {\tt
crlochri@olemiss.edu}} \vspace{.1cm}  \\  University of Mississippi, Department of Mathematics  \vspace{.4cm} }

\begin{document}
\maketitle

\begin{abstract}
We study models for a directed polymer in a random environment (DPRE) in which the polymer traverses a hierarchical diamond graph and the random environment is defined through random variables attached to the vertices. For these models, we prove a distributional limit theorem for the partition function in a limiting regime wherein the system grows as the coupling of the polymer to the random environment is appropriately attenuated.  The sequence of diamond graphs is determined  by a choice of a branching number $b\in \{2,3,\ldots\}$ and segmenting number $s\in \{2,3,\ldots\}$, and  our focus is on the critical case of the model where $b=s$.  This extends recent work in the critical case of analogous models with disorder variables placed at the edges of the graphs rather than the vertices. 
\end{abstract}

\section{Introduction}\label{SectionIntro}

A \textit{directed polymer in a random environment} (DPRE) is a probabilistic model motivated by statistical mechanics that is mathematically defined as a random measure on pathways traversing some discrete or continuous spatial structure. The most studied class of DPRE models begins with an $N$-step $d$-dimensional simple symmetric random walk, in other terms the uniform probability measure on maps $p:\{0,1,...,N\}\rightarrow \Z^{d}$ satisfying $p(0)=0$ and $\|p(j)-p(j-1)\|_2=1$ for each $j\geq 1$. Then a family of i.i.d.\ random variables $\{\omega_{j,a}\}$ indexed by coordinates ${(j,a) \in \{1,...,N\} \times \Z^d}$  and an inverse temperature parameter $\beta \geq 0$ are used to define a random measure $\mathbf{M}_{\beta,N}^{\omega}$ on these nearest-neighbor paths through 
\begin{align}\label{M}
\mathbf{M}_{\beta,N}^{\omega}(p)\, := \,\textup{exp}\big\{\beta H^\omega_N(p) -N\lambda(\beta) \big\}   \,, 
\end{align}
where $H^\omega_N(p):= \sum_{j=1}^{N}  \omega_{j,p(j)}$ is called the \textit{path energy}, and $\lambda(\beta):=\log\big(\mathbb{E}\big[ e^{\beta\omega_{j,a}}  \big]\big)  $ is the cumulant generating function of $\omega_{j,a}$.  We denote the total mass of the random path measure $\mathbf{M}_{\beta,N}^{\omega}$ by $W_{\beta,N}^{\omega}$, which we refer to as the \textit{partition function}. The collection $\{\omega_{j,a}\}$ comprises the random environment, encoding localized impurities that generate a reweighing of paths through the Gibbsian formalism~(\ref{M}).  In the standard case, the random variables $\omega_{j,a}$ are assumed to have mean zero, variance one, and finite exponential moments.  The parameter $\beta$ effectively determines the coupling strength of the polymer to its random environment.  For fixed values of  $d$ and  $\beta$, the system is said to be \textit{strongly disordered} if the presence of the random environment has a marked effect on the behavior of the polymer as the size of the system scales up ($N\gg1 $); otherwise, the system is termed \textit{weakly disordered}.  It is known that when $d\leq 2$ these DPRE models are \textit{disorder relevant}, meaning that strong disorder occurs for all  $\beta>0$. Conversely, DPRE models are \textit{disorder irrelevant} when $d\geq 3$; that is, for small enough fixed values of $\beta>0$ the environmental disorder has a diminishing effect as $N\uparrow \infty$.  See the monograph~\cite{Comets} by Comets for an account of important developments in the theory of DPRE models and the monograph~\cite{Giacomin} by Giacomin for a discussion of disorder relevance versus irrelevance in the context of pinning models.

One natural direction within the study of disorder relevant models is to consider weak coupling limits in which the length $N$ of the polymer grows to $\infty$ as the inverse temperature parameter $\beta \equiv \beta_N$ vanishes under an appropriate scaling dependence on $N$.  In the article~\cite{alberts}, Alberts, Khanin, and Quastel introduced a scaling limit of this type in the case $d=1$, wherein the inverse temperature is scaled as $\beta_N:=\hat{\beta} N^{-1/4} $ for some value of the parameter $\hat{\beta}\in [0,\infty)$.  The authors referred to this weak coupling limit as the \textit{intermediate disorder regime} because it explores a vanishing window of behavior as $N \uparrow \infty$ between the trivial weak disorder case of $\beta=0$ and the strong disorder that prevails when $\beta$ is held fixed with any strictly positive value.  In particular, \cite{alberts} proved that the partition function $W_{\beta_N,N}^{\omega}$ converges in distribution  as $N\uparrow \infty$  to a nontrivial  limit law, $W_{\hat{\beta}}$.    The family of limit laws $\{W_{\hat{\beta}}\}_{\hat\beta\in[0,\infty)}$ have mean one and undergo a transition from weak disorder to strong disorder as the parameter $\hat{\beta}$ increases from $0$ to $\infty$ in the sense that $W_{0}=1$ almost surely and $W_{\hat{\beta}}$ converges in distribution to $0$ as $\hat{\beta}\uparrow \infty$.  For each $m>1$, the moment $\mathbb{E}\big[ W_{\hat{\beta}}^m\big]$ is finite for all $\hat{\beta}$ but diverges to $\infty$ as $\hat{\beta}\uparrow \infty$.

Since the DPRE models are disorder irrelevant when $d\geq 3$, the case $d=2$ is the borderline of disorder relevance. For $d=2$, the pursuit of an intermediate disorder regime result analogous to~\cite{alberts} introduces nontrivial technical and conceptual difficulties. The subtlety of this case begins with finding an appropriate choice of vanishing inverse temperature scaling $\beta_N$ with $N\uparrow \infty$. In the article~\cite{CSZ1}, Caravenna, Sun, and Zygouras proved that when $\beta_N=\frac{\hat{\beta}}{\log^{1/2} N }\big(1+\mathit{o}(1)\big)$ the partition functions $W_{\beta_N,N}^{\omega}$ have the convergence in distribution
\begin{align}\label{IDR_d=2}
W_{\beta_N,N}^{\omega} \hspace{.5cm} \Longrightarrow \hspace{.5cm} W_{\hat{\beta}}\,:=\,\begin{cases}  \,\textup{Lognormal}\Big(-\frac{1}{2}\sigma_{\hat{\beta}}^2 ,\,\sigma^2_{\hat{\beta}}\Big) & \hspace{.3cm}  \hat{\beta} < \sqrt{\pi} \,,    \\ \, 0    & \hspace{.3cm}  \hat{\beta} \geq {\sqrt{\pi}}  \,, \end{cases}   
\end{align}
for 
$\sigma_{\hat{\beta}}^2 := \log\frac{\pi}{\pi-\hat{\beta}^2} $.  Hence,  there is a critical point, $\hat{\beta}_c := \sqrt{\pi}$, for the  scaling parameter $\hat{\beta}$  beyond which the partition function $W_{\beta_N,N}^{\omega}$ exhibits strong disorder in the limit $N\uparrow \infty$.  As the parameter $\hat{\beta} $ approaches $\sqrt{\pi}$ from below, the lognormals $W_{\hat{\beta}}$ converge in distribution to zero while maintaining mean one and having higher moments that diverge to $\infty$ (in particular, this is found in the second moment of $W_{\hat{\beta}}$, which is  $\frac{\pi}{\pi-\hat{\beta^2}}$). Thus, the limit distribution $W_{\hat{\beta}}$ is weakly continuous in the parameter $\hat{\beta}$, but the higher moments have an infinite discontinuity at $\hat{\beta}_c$.

 In~\cite{CSZ2,CSZ4} the same authors introduced a more refined  scaling procedure that magnifies a vanishing region around the critical point $\hat{\beta} = {\sqrt{\pi}}$ arising in~(\ref{IDR_d=2}) and involves a ``randomized" starting point for the polymer:
\begin{itemize}
    \item For each $x\in \mathbb{Z}^2$, let $W_{\beta,N}^{\omega}(x) $ denote the partition function analogous to $W_{\beta,N}^{\omega}$ for polymers that begin at $x$.  
    
    \item For any continuous function $\psi: \R^2\rightarrow \R$ with compact support, define the mollified partition function $ W_{\beta,N}^{\omega}(\psi) :=\frac{1}{N} \sum_{x\in \Z^2 }\psi\big(\frac{x}{\sqrt{N}}\big) W_{\beta,N}^{\omega}(x)$.

    \item For a parameter value $\hat{\vartheta}\in \R $, let $\big(  \beta_{N, \hat{\vartheta}}\big)_{N\in \mathbb{N}}$ be a sequence in  $(0,\infty) $ with the large-$N$ asymptotics
\begin{align}\label{BetaCSZ}
\beta_{N, \hat{\vartheta}}\,=\,\frac{\sqrt{\pi}}{\log^{\frac{1}{2}}N}\left(1\,-\,\frac{\kappa_3\sqrt{\pi}}{2\log^{\frac{1}{2}} N}\,+\,\frac{\hat{\vartheta}+\pi\big(\frac{5}{4}\kappa^2_3-\frac{1}{2}-\frac{7}{12}\kappa_4\big)}{2\log N}\right)\,+\,\mathit{o}\bigg(\frac{1}{\log^{\frac{3}{2}} N}\bigg)\,,
\end{align}
in which $\kappa_3 :=\mathbb{E}[\omega^3]$ and $\kappa_4 :=\mathbb{E}[\omega^4]-3$ are the third and fourth cumulants of the disorder variables.\footnote{The parameter $\hat{\vartheta}$ is related to $\vartheta$ in~\cite{CSZ4} through  $\vartheta = \hat{\vartheta} +\gamma_{\textup{EM}}+\log 16-\pi$, where $\gamma_{\textup{EM}}$ is the Euler-Mascheroni constant.} 
\end{itemize}
The inverse temperature asymptotic~(\ref{BetaCSZ}) satisfies $\beta_{N,\hat{\vartheta}}=\frac{\sqrt{\pi}}{\log^{1/2} N }\big(1+\mathit{o}(1)\big)$ and is chosen so that the variance of the mean one random variable $ \textup{exp}\big\{\beta_{N,\hat{\vartheta}}\,\omega -\lambda(\beta_{N,\hat{\vartheta}}) \big\}  $ has the following large-$N$ form:  $ \frac{\pi}{\log N}+\frac{\pi \hat{\vartheta}}{\log^2 N}\big(1+\mathit{o}(1)\big)$.
The  above scaling regime  is closely related to the critical  scaling limit for the $2d$ stochastic heat equation (SHE) introduced by Bertini and Cancrini in~\cite{BC}, for which Gu, Quastel, and Tsai provided a functional analysis method for handling the convergence of the positive integer moments in~\cite{GQT}; see also the related work~\cite{Chen} by Chen. The article~\cite{CSZ4} proved tightness of the sequence of random variables $\big(W_{\beta_{N,\hat{\vartheta}},N}^{\omega}(\psi)\big)_{N\in \mathbb{N}} $ for any test function $\psi\in C_c\big(\R^2\big)$---which can be interpreted in a broader sense as tightness for a sequence of random $\sigma$-finite Borel measures  on $\R^2$---and that all distributional limits of the random measures have the same covariance structure, depending on the parameter $\hat{\vartheta}\in \R$. In their more recent  works~\cite{CSZ5,CSZ6},  Caravenna, Sun, and Zygouras   deduced the  uniqueness of these distributional limits and showed that the limiting random measure law is not a Gaussian multiplicative chaos.

In this article, we study a family of DPRE models defined on \textit{diamond hierarchical graphs}, which we define in Section~\ref{SecDHG}.  Hierarchical graphs were introduced in physics literature as a reduced-complexity medium for studying various phenomena; see for instance~\cite{Berker,Cook,Kaufman} on Ising/Potts models, \cite{Derrida} on directed polymers, and~\cite{Derrida2} on wetting transitions. Mathematicians subsequently adopted the hierarchical setting to explore various probabilistic and dynamical systems topics in mathematical physics,~\cite{Pavel2,Pavel,Wehr,Goldstein,GLT,Lacoin,Hambly,Pavel1,Ruiz} being a non-exhaustive list of such works.   Although the hierarchical models are artificial, they can provide insights leading to results on standard models. For instance, Lacoin's work in~\cite{Lacoin0} on the free energy behavior at high temperature for rectangular lattice polymers in the $d=1$ and $d=2$ cases took partial inspiration from his prior work with Giacomin and Toninelli on hierarchical pinning models in~\cite{GLT}. For a fixed \textit{branching} parameter $b \in \{2,3,\ldots\}$ and  \textit{segmenting} parameter $s\in \{2,3,\ldots\}$, the diamond graphs are recursively constructed using a graphical embedding procedure, generating a nexus of directed paths between two opposing nodes.  In~\cite{Lacoin}, Lacoin and Moreno studied the phase diagram for diamond graph DPRE models as a function of the parameters $b$ and $s$ and the inverse temperature $\beta$, showing that there is a rough analogy in the disorder behavior between the following cases for rectangular lattice polymers:
\begin{align*}
    & b<s  \hspace{.7cm} \longleftrightarrow  \hspace{.7cm}    d=1\,,     \\
    & b=s  \hspace{.7cm} \longleftrightarrow  \hspace{.7cm}    d=2\,,     \\
    & b>s  \hspace{.7cm} \longleftrightarrow  \hspace{.7cm}    d \geq 3\,. 
\end{align*}
 In particular, disorder relevance holds for the diamond graph DPRE models only when $b\leq s$, where $b=s$ is the marginal case.  These observations  hold whether the disordered environment is formulated through attaching disorder variables to the vertices  of the diamond graphs or to their edges. The disorder relevance of the subcritical case $b<s$ and the critical case $b=s$ allows for the possibility of performing an intermediate disorder regime analysis comparable with~\cite{alberts}. Such an analysis was carried out for $b<s$ in~\cite{US}, yielding a distributional limit theorem for the partition functions analogous to that in~\cite{alberts} and covering models with either vertex  disorder or edge disorder. For models with edge disorder, the article~\cite{Clark2} developed this analysis further to include a distributional limit theorem for the $b=s$ case within a critical scaling window similar to that discussed above for the $d=2$ rectangular lattice polymers. 
 
 The goal of this text is to extend the distributional limit result for the critical ($b=s$) edge-disorder models in~\cite{Clark2} to the case of vertex  disorder, which some readers will find to be a more natural convention.  The distributional recurrence relations for the edge-disorder partition functions are homogeneous in a sense that  the vertex-disorder counterparts are not; see the remark following~(\ref{PartHierSymmII}) in the next section. As a consequence, the fine-tuning of the inverse temperature asymptotics to induce the  convergence of the variance of the partition function $\mathbf{W}_r^{b,b}$---a precondition for formulating  the limit theorem---is  more intricate for the vertex-disorder model. Our analysis proceeds by showing that the vertex-disorder model can be approximated by a smaller related edge-disorder model through removing a portion of the disorder variables, yielding  a negligible error. Our analysis requires one  particularly delicate step (Lemma~\ref{LemMtilde}), which is to determine the appropriate inverse temperature scaling of the original vertex model from the needed variance scaling of its edge-disordered reduction.

After establishing notation and recalling a result from~\cite{Clark2}, the following section concludes with a precise statement of our main result, Theorem~\ref{ThmMain}, whose proof comes at the end of Section~\ref{SecMainThmOutline}. The proof assumes three technical lemmas, which we prove in Section~\ref{SecThreeLemmas}.

\section{The setup and a statement of the main result} \label{SectionMainResult}

In Section~\ref{SecSubMain} below, we present our main result, a distributional limit theorem for partition functions defined from hierarchical DPRE models. As a preliminary, in Section~\ref{SecDHG} we construct the family of diamond graphs, each providing a structure on which to define a space of interweaving directed pathways. Section~\ref{SecGibbs} introduces the formalism for our disorder model, which is a  measure on directed paths (polymers) randomized through a Gibbsian multiplicative noise factor depending on a collection of i.i.d.\ random variables attached to the vertices of the diamond graph. Although we concentrate on the critical case ($b=s$), Section~\ref{SecSubCri} includes the statement from~\cite{US} of a subcritical ($b<s$) analog  of our main result for the purpose of comparison. For some motivating context, Section~\ref{SecSca} recalls a previous incomplete result on the critical case.

\subsection{Construction of the diamond hierarchical graphs}\label{SecDHG}

We begin by recursively defining a sequence of graphs $\big(D_n^{b,s}\big)_{n\in\N_0}$, where $\N_0$  denotes the set of nonnegative integers. Let $D_0^{b,s}$ denote the graph formed by two \textit{root vertices} $A$ and $B$ with a single edge between them, and let $D_1^{b,s}$ be the graph consisting of $b$ parallel branches connecting $A$ to $B$, each branch having $s$ edges running in series. For every remaining $n \in \N$, we construct $D_n^{b,s}$  by substituting each edge $h$ of the graph $D_{1}^{b,s}$ with an embedded copy of $D_{n-1}^{b,s}$, where the root vertices $A$ and $B$ of the embedded graph take the positions of the vertices incident to $h$; see Figure~\ref{Fig1}.
\begin{figure}[hbt!]
\centering
\includegraphics[scale=.6]{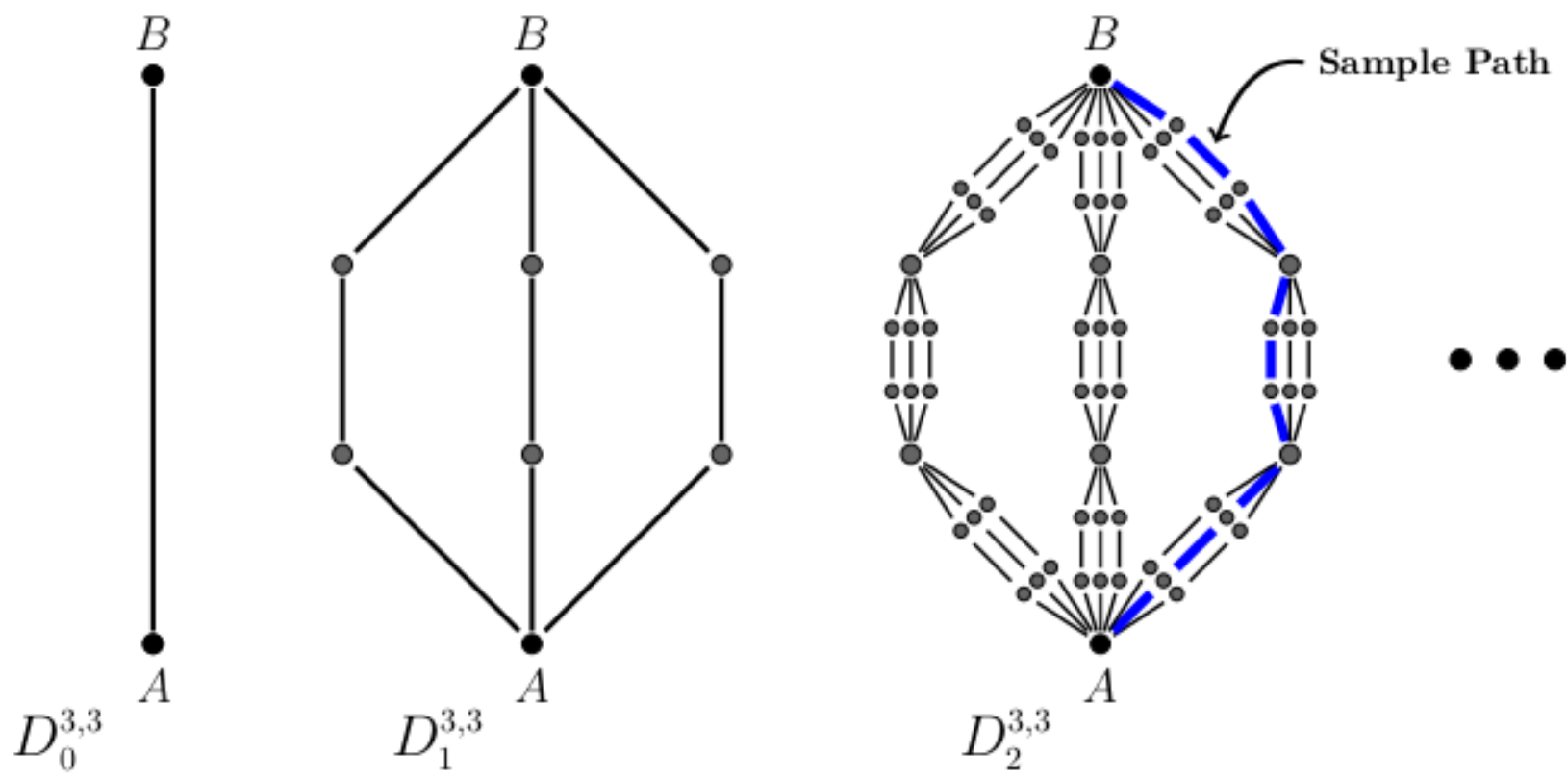}
\begin{minipage}{.81\textwidth}
\caption{\label{Fig1} Displayed above are the first three recursively-defined diamond graphs with $\bbf=\sbf=3$.  A path $p\in \Gamma_2^{3,3}$ is highlighted in the diagram of the second-generation diamond graph. }
\end{minipage}
\end{figure}
For $n \in \N$, let $V_n^{b,s}$ denote the set of non-root vertices on the diamond graph $D_n^{b,s}$, and let $E_n^{b,s}$ denote the set of edges on the diamond graph $D_n^{b,s}$. Henceforth, the term \textit{vertex} will refer only to non-root vertices. Note that vertex sets of the diamond graphs have an embedding property in the sense that $V_{n-1}^{b,s}$ is canonically identifiable with a subset of $V_n^{b,s}$ for each $n \in \N$. Through this interpretation, we refer to $V_n^{b,s} \backslash  V_{n-1}^{b,s}$  as the set of \textit{generation-$n$ vertices}. In other words, the generation-$n$ vertices are those that appear in $D_n^{b,s}$ but not  $D_{n-1}^{b,s}$. For  $k,n \in \N$ with $k \leq n$, the edge set $E_k^{b,s}$ does not have such a direct canonical embedding in $E_n^{b,s}$. However, each edge in $E_k^{b,s}$ is canonically identifiable with  a subgraph of $D_n^{b,s}$ that is isomorphic to $D_{n-k}^{b,s}$, and the  edge sets  of these embedded copies of $D_{n-k}^{b,s}$ form a partition of $E_{n}^{b,s}$.  This hierarchical structure implies that $\big|E_n^{b,s}\big|=(bs)^n$ and $\big|V_n^{b,s} \backslash V_{n-1}^{b,s}\big|= b(s-1)(bs)^{n-1}$.
 
Each diamond graph $D_n^{b,s}$ determines a set  $\Gamma_n^{b,s}$ of \textit{directed paths} from $A$ to $B$, which are maps $p:\{1,...,s^n-1\}\rightarrow V_n^{b,s}$ such that the vertex $p(1)$ is adjacent to $A$, the vertex $p(s^n-1)$ is adjacent to $B$, and the vertices $p(k-1)$ and $p(k)$ are adjacent for all $k \in \{2,...,s^n-1\}$. This definition ensures that each directed path $p \in \Gamma_n^{b,s}$ progresses monotonically from $A$ to $B$.

\subsection{A random Gibbsian measure on directed paths}\label{SecGibbs}

Let $\{\omega_a\}_{a \in V_n^{b,s}}$  be a family of i.i.d.\ centered random variables with variance one such that $\E[e^{\beta \omega_a}]<\infty$ for all $\beta \in [0,\infty)$. For each path $p \in \Gamma_n^{b,s}$, we define the \textit{energy} of $p$ by
$$H_n^{\omega}(p)\,:= \,\sum_{a\in p} \omega_a\,,$$
where $a\in p$ means  that the vertex $a$ is in the range of the path $p$.\footnote{Recall that, by convention, the root nodes $A$ and $B$ are not elements of the vertex set $V_n^{b,s}$.} From this, we define the following random measure on $\Gamma^{b,s}_n$ for $n\geq1$ and  a fixed  inverse temperature value $\beta \in [0, 
\infty)$:
\begin{align*}
\mathbf{M}^{\omega}_{\beta, n}(p) \,   := \,\frac{1}{\big|\Gamma^{b,s}_n\big|} \textup{exp}\Big\{\beta H^\omega_n(p) -(s^n-1)\lambda(\beta) \Big\}\,=\,\frac{1}{\big|\Gamma^{b,s}_n\big|}\prod_{a\in p}\textup{exp}\big\{\beta \omega_a-\lambda(\beta) \big\} \, ,
\end{align*}
recalling that $\lambda(\beta):=\log \big(\mathbb{E}\big[e^{\beta\omega_a} \big ]\big) $.
Note that the above is a uniform probability measure on the path space when $\beta = 0$. The \textit{partition function}, $W_{n}^\omega(\beta)$, is the random variable defined by the total mass of $\mathbf{M}^{\omega}_{\beta, n}$, meaning
\begin{align}\label{PF}
    W_{n}^\omega(\beta)\,:=\, \mathbf{M}^{\omega}_{\beta, n}\big(\Gamma^{b,s}_n\big)\,=\,  \frac{1}{\big|\Gamma_{n}^{b,s}\big|  }\sum_{p\in \Gamma_{n}^{b,s}  }\mathbf{M}^{\omega}_{\beta, n}(p) \, .  
\end{align}
When $n=0$, we define $W_{0}^\omega(\beta):=1$ since $\Gamma_0^{b,s} = \emptyset$. The hierarchical symmetry resulting from the embedding procedure used in the construction of the diamond graphs implies the distributional recurrence relation for the partition functions given below:
\begin{align}\label{PartHierSymmII}
W_{n+1}^{\omega}(\beta)\,\stackrel{d}{=}\, \frac{1}{b}\sum_{1\leq i \leq b} \Bigg(\prod_{1\leq j \leq s}W_{n}^{(i,j)}(\beta) \Bigg)\Bigg( \prod_{1\leq \ell \leq s-1} \textup{exp}\left\{ \beta \omega_{i,\ell}  -\lambda\big(\beta \omega_{i,\ell}\big)\right\} \Bigg) ,
\end{align}
in which  $\big\{W_{n}^{(i,j)}(\beta)\big\}_{i,j}$ and $\{\omega_{i,\ell}\}$ are respectively families of independent copies of the random variables $W_{n}^{\omega}(\beta)$ and $\omega_a$, and the two collections are independent. In the above, each $W_{n}^{(i,j)}(\beta) $ corresponds to a subgraph of  $D_{n+1}^{b,s}$ isomorphic to $D_n^{b,s}$, and each 
$\textup{exp}\big\{ \beta \omega_{i,\ell}  -\lambda(\beta \omega_{i,\ell})\big\}$ corresponds to a generation-1 vertex of $D_{n+1}^{b,s}$. Notice that the recurrence relation for the edge-disorder model is simpler than~(\ref{PartHierSymmII}) in that the second product on the right side of~(\ref{PartHierSymmII}) is not present.

For $k\in \mathbb{N}_{0}$ and $\beta>0$, let  ${\varrho}_{k}(\beta)$ denote the variance of the partition function ${W}^{\omega}_{k}(\beta)$.  As a consequence of the distributional identity~(\ref{PartHierSymmII}), the sequence  of variances $\big(\varrho_{k}(\beta)\big)_{k\in \mathbb{N}_0}$ satisfies the recursive equation
\begin{align}\label{RecEqVar}
   \varrho_{k+1}(\beta)\,=\, M_{V}^{b,s}\big(\varrho_{k}(\beta)\big)  \hspace{1cm}\text{with}\hspace{1cm}{\varrho}_{0}(\beta)\,=\,0\,,
\end{align}
where the map ${M}_{V}^{b,s}:[0,\infty)\rightarrow [0,\infty)$ is defined by
\begin{align*}
   M_{V}^{b,s}(x) \,:=\,\frac{1}{b}\Big[ (1+x)^s\big(1+ V \big)^{s-1}    \,-\,1   \Big] \hspace{.5cm}\text{for}\hspace{.5cm} V\,\equiv\, V(\beta)\,:=\,\textup{Var}\left( \textup{exp}\big\{\beta \omega -\lambda(\beta) \big\} \right)     \,. 
\end{align*}
Note that when $V=0$, the map $M_{V}^{b,s}(x)$ reduces to $M^{b,s}(x):=\frac{1}{b}\big[(1+x)^s-1\big]$, which has  the  $0<x\ll 1$ asymptotics
\begin{align}\label{Repelling}
M^{b,s}(x)\, =\,\begin{cases}  \frac{s}{b}x+\mathit{O}\big(x^2\big) &  b\neq s \,,  \\   x+\frac{b-1}{2}x^2 +\mathit{O}\big(x^3\big)  & b=s \,.
\end{cases}
\end{align}
Thus, the fixed point $x=0$ for the variance map $M^{b,s}(x)$ is repelling if and only if $b\leq s$, but it is merely marginally repelling when $b=s$.

\subsection{Previous result on the $\mathbf{b<s}$ case}\label{SecSubCri}

Next we turn our discussion to scaling limits where the  generation parameter $n\in \mathbb{N}$ of the diamond graphs grows while the inverse temperature  $\beta\equiv\beta_n$ vanishes at a rate such that the partition function $W_{n}^{\omega}(\beta_n)$ converges in law to a nontrivial  limit. This is only possible  in the cases $b < s$ and $b=s$, where the diamond graph DPRE model is disorder relevant. The following limit theorem is from~\cite[Theorem 2.1]{US}, and the limit law $\mathbf{W}_{r}^{b,s}$ appearing in its statement was shown in~\cite{Clark3} to be the partition function for a  continuum DPRE model analogous to that introduced  in~\cite{alberts2} for the continuum limit of the $d=1$ rectangular lattice model. That is, there is a canonical family of  random measure laws $\{M_{r}^{b,s}\}_{r\in (0,\infty)}$ having total mass equal in distribution to $\mathbf{W}_r^{b,s}$ and acting on a space $\Gamma^{b,s}$ of continuum pathways across a diamond fractal that arises as a ``limit" of the diamond graphs $D_n$ as $n\uparrow \infty$.

\begin{theorem}\label{ThmOLDb<s}
Let $\widehat{\beta} > 0$, and define $\beta_n := \widehat{\beta}  (\frac{b}{s})^{n/2}$. As $n \uparrow \infty$, we have convergence in distribution 
$$W_{n}^{\omega}(\beta_n)\quad  \Longrightarrow \quad  \mathbf{W}_{\widehat{\beta} ^2 \frac{s-1}{s-b}}^{b,s} \,, $$
where the  family of distributions $\big\{\mathbf{W}_{r}^{ b,s}\big\}_{r\in [0,\infty)}$ has properties (I)--(III) below.
\begin{enumerate}[(I)]
    \item $\mathbf{W}_{r}^{b,s}$ has mean 1 and variance $R_{b,s}(r)$ for a function $R_{b,s}: [0,\infty) \rightarrow [0,\infty)$ satisfying 
    \begin{align*}
        R_{b,s}\left(\frac{s}{b}r\right) \, = \, \frac{1}{b}\Big[\big(1+ R_{b,s}(r)\big)^s -1\Big] \hspace{1cm} and \hspace{1cm} \lim_{r \downarrow 0} \frac{R_{b,s}(r)}{r}=1\,.
    \end{align*}

 \item If $X_r$ is a random variable with distribution $\mathbf{W}_{r}^{b,s}$, then $\frac{X_r-1}{\sqrt{r}}$ converges in distribution to $\mathcal{N}(0,1)$  as $r \downarrow 0$. 
 
    \item If $\big\{X_r^{(i,j)}\big\}_{1\leq i\leq b,\,  1\leq j\leq s}$ is a family of independent random variables with distribution $\mathbf{W}_{r}^{b,s}$, then there is equality in distribution \vspace{-.1cm}
    $$\mathbf{W}_{\frac{s}{b}r}^{b,s}\, \stackrel{d}{=} \, \frac{1}{b}\sum_{1 \leq i \leq b} \prod_{1 \leq j \leq s} X_r^{(i,j)}\,.$$

\end{enumerate}
\end{theorem}
 Statement (III) derives as a limit of the distributional recurrence relation~(\ref{PartHierSymmII}), and the variance relation in (I) follows from it,  assuming that the second moments are finite.

The geometric form with  common ratio $\big(\frac{b}{s}\big)^{1/2}$   for the inverse temperature  $\beta_n$ is a reasonable choice considering the linear repelling~(\ref{Repelling})  of the variance map $M^{b,s}$ near $x=0$ when $b<s$.  In the case $b=s$, the inverse temperature scaling  used above reduces merely to $\widehat{\beta}$, which fails to vanish as $n\uparrow \infty$.  Thus, a different choice of inverse temperature scaling $\beta_n$ is needed to obtain a distributional convergence result analogous to Theorem~\ref{ThmOLDb<s} when $b=s$. 

\subsection{Inverse temperature scaling in the $\mathbf{b=s}$ case}\label{SecSca}
 The proposition below from~\cite[Theorem 2.5]{US} examines the large-$n$ behavior of the partition function $ W_{n}^{\omega}(\beta_n )$  in the critical case $b=s$  when the inverse temperature is taken to be of the form $\beta_n=\frac{\widehat{\beta} }{n}$. There is a critical point in the behavior at $\widehat{\kappa}_b := \frac{ \pi\sqrt{b}   }{\sqrt{2}(b-1)   } $ for the parameter $\widehat{\beta} \in [0,\infty)$, which indicates that a more refined inverse temperature scaling is required to achieve an analog of Theorem~\ref{ThmOLDb<s}.  
\begin{proposition}\label{RemarkBetaScaleII} For $b\in \{2,3,\ldots\}$, define the map $\upsilon_b:[0,\widehat{\kappa}_b)\rightarrow [0,\infty)$ by $\upsilon_b(\widehat{\beta}):=\widehat{\beta}\frac{ \sqrt{2} }{\sqrt{b}  }\tan\big(\frac{\pi}{2}\frac{\widehat{\beta}}{ \widehat{\kappa}_b}   \big)$.  When $b=s$, the partition function $ W_{n}^{\omega}\big(\frac{\widehat{\beta}}{ n}\big)$ has the large-$n$  distributional behaviors listed below, depending on the parameter $\widehat{\beta}\geq 0$. 
\begin{itemize}
\item When $\widehat{\beta} <\widehat{\kappa}_b $, the variance of $W_{n}^{\omega}\big(\frac{\widehat{\beta}}{ n}\big)$ vanishes as  $\frac{1}{ n^2}\big(\upsilon_b(\widehat{\beta})+\mathit{o}(1)\big)$ with large $n$, and there is convergence in distribution
 $$n\bigg(W_{n}^{\omega}\bigg(\frac{\widehat{\beta}}{ n}\bigg)-1\bigg)\hspace{.5cm} \Longrightarrow \hspace{.5cm}  \mathcal{N}\Big(0, \upsilon_b\big(\widehat{\beta}\big)\Big)\,. $$
\item When $\widehat{\beta} =\widehat{\kappa}_b $,  the variance of $W_{n}^{\omega}\big(\frac{\widehat{\beta}}{ n}\big)$ vanishes as  $\frac{1}{ \log n}\big(\frac{6}{b+1}+\mathit{o}(1)\big)$  with large  $n$, and there is convergence in distribution
 $$\sqrt{\log n}\bigg(W_{n}^{\omega}\bigg(\frac{\widehat{\beta}}{ n}\bigg)-1\bigg) \hspace{.5cm} \Longrightarrow \hspace{.5cm} \mathcal{N}\Big(0, \frac{6}{b+1} \Big)\,.$$
\item When $\widehat{\beta} >\widehat{\kappa}_b$, the variance of  $W_{n}^{\omega}\big(\frac{\widehat{\beta}}{ n}\big)$ diverges to $\infty$ as $n$ increases. 
 \end{itemize}
\end{proposition}

Let $\widehat{\kappa}_b$ be as above, and define  $\eta_{b}:=\frac{b+1}{3(b-1) }$, $\tau:=\mathbb{E}\big[\omega_a^3\big]$, and $\varsigma_b := (\log\frac{\pi}{2}+2)\eta_b  $.  The distributional limit theorem in the next subsection uses an inverse temperature  scaling $\big(\beta_{n,r}^{(b)}\big)_{n\in \mathbb{N}}$ with the large-$n$ asymptotic form
\begin{align}\label{BetaForm2}
\beta_{n,r}^{(b)}\,=\,\frac{\widehat{\kappa}_b}{n} \bigg(1\,+\,\frac{  \eta_b \log n  }{ n }\,+\,\frac{r-\varsigma_b-\widehat{\kappa}_b\frac{\tau}{2} }{ n }\bigg)\,+\,\mathit{o}\Big( \frac{1}{n^2} \Big)
\end{align}
for a fixed value of the parameter $r\in \R$.  Note that  $\beta_{n,r}^{(b)}$ has the form $\frac{\widehat{\beta}}{n}\big(1+\mathit{o}(1)\big) $ for the critical value $\widehat{\beta}= \widehat{\kappa}_b $.

\subsection{Main result}\label{SecSubMain}
The following theorem is the counterpart to~\cite[Theorem 2.7]{Clark2}, where the disorder variables $\{\omega_a\}$ are attached to the edges of the graphs rather than the vertices, and the proof resides in Section~\ref{SecSiteDisorder}. The articles~\cite{Clark4,Clark5} study critical continuum DPRE models that correspond to the limiting partition function laws $\{\mathbf{W}_{r}^{(b)}\}_{r\in\R}$ and are analogous to the subcritical continuum polymer models mentioned before Theorem~\ref{ThmOLDb<s}.

\begin{theorem}\label{ThmMain}
Fix $b\in \{2,3,\ldots\}$ and $r\in \R$, and assume $s=b$.  If the sequence $\big(\beta_{n,r}^{(b)}\big)_{n\in \mathbb{N}}$ has the asymptotic form in~(\ref{BetaForm2}), then there is convergence in distribution as $n\uparrow \infty$ 
$$W_{n}^{\omega}\big(\beta_{n,r}^{(b)}\big) \quad \Longrightarrow \quad \mathbf{W}_{r}^{(b)} \, , $$
where the family of  distributions $\big\{\mathbf{W}_{r}^{(b)}\big\}_{r\in \R}$ uniquely satisfies (I)--(IV) below.
\begin{enumerate}[(I)]
    \item $\mathbf{W}_{r}^{(b)}$ has mean one and variance $R_{b}(r)$ for a function $R_{b}: \R \rightarrow (0,\infty)$ satisfying 
    \begin{align*}
        R_{b}(r+1)\, =\, \frac{1}{b}\left[\big(1+ R_{b}(r)\big)^b -1\right] \hspace{.5cm} and \hspace{.5cm} R_{b}(r)\,=\,\frac{\kappa_b^2}{-r}\bigg(1+\frac{ \eta_b \log(-r) }{-r   }\bigg)   +\mathit{O}\bigg( \frac{\log^2(-r) }{r^3   }  \bigg)\,,
    \end{align*}
    where $\kappa_b := (\frac{2}{b-1})^{1/2} $ and the asymptotic above occurs in the limit $r\downarrow -\infty$.

\item  For each $m\in \{3,4,\ldots\}$, the $m^{th}$ centered moment  of $\mathbf{W}_{r}^{(b)}$ is finite for all $r\in \R$.  Moreover, the $m^{th}$ centered moment vanishes in proportion to $(-r)^{-\lceil  m/2\rceil  } $ as $r\downarrow -\infty$, and diverges to $\infty$ as $r\uparrow \infty$.

 \item If $X_r$ is a random variable with distribution $\mathbf{W}_{r}^{(b)}$, then $\sqrt{-r} (X_r-1) $ converges in distribution to $\mathcal{N}\big(0,\kappa_b^2 \big)$ as $r \downarrow -\infty$.

 \item If $\big\{X_r^{(i,j)}\big\}_{1\leq i,j\leq b}$ is a family of independent random variables with distribution $\mathbf{W}_{r}^{(b)}$, then there is equality in distribution
    $$\mathbf{W}_{r+1}^{(b)} \,\stackrel{d}{=} \,\frac{1}{b}\sum_{1 \leq i \leq b} \prod_{1 \leq j \leq b} X_r^{(i,j)}\,. $$ 
    
\end{enumerate}
\end{theorem}
Statement (I) implies that $\mathbf{W}_{r}^{(b)}$ converges weakly to one as $r\downarrow -\infty$, and $\mathbf{W}_{r}^{(b)}$ converges weakly to zero as $r\uparrow \infty$ by~\cite[Proposition 5.1]{Clark5}.  Thus, the family of limit  laws $\big\{\mathbf{W}_{r}^{(b)}\big\}_{r\in \R}$ undergoes a  transition from weak disorder to strong disorder as the parameter $r\in \R$ moves upwards from $-\infty$ to $\infty$. \vspace{.1cm}

The above theorem can be used to formally derive a high-temperature ($0<\beta\ll 1$) asymptotic for the free energy
$$ F(\beta)\,:=\,\lim_{n \to\infty}\frac{1}{b^n}\mathbb{E}\big[ \log W_n^\omega(\beta) \big] \,. $$
By inverting the inverse temperature scaling~(\ref{BetaForm2}) to write $n$ in terms of $\beta\equiv \beta_{n,r}^{(\beta)}$ for a fixed large value of $r>0$, we obtain that for small $\beta$ 
$$   F(\beta)\,=\,\mathbf{c} \beta^{\epsilon_b}\,b^{-\frac{\widehat{\kappa}_b}{\beta}}  \big(1+\mathit{o}(1)\big) \,, $$
where $\epsilon_b:=\eta_b\log b  $, and the constant $\mathbf{c}<0$ is defined by $\mathbf{c}:= b^{-\eta_b\log \widehat{\kappa}_b +\varsigma_b +\widehat{\kappa}_b\frac{\tau}{b}  } \lim_{r\uparrow \infty}\frac{1}{b^r}\mathbb{E}\big[\mathbf{W}_r^{(b)}\big]$.

\section{The $\boldsymbol{\mathcal{Q}}$ map, three  lemmas, and a  proof of the main result}\label{SecMainThmOutline}

The discussion in this section ends with the proof of Theorem~\ref{ThmMain}, which requires a distributional convergence theorem from \cite{Clark2} and the technical results in Lemmas~\ref{LemCond},~\ref{LemMtilde}, and~\ref{LemmaHM} that will be proved in Section~\ref{SecThreeLemmas}. To apply the convergence theorem from~\cite{Clark2}, we show that vertices of generation lower than $\log n$ can be removed from the partition function $W_{n}^{\omega}(\beta)$ with a negligible error (see Lemma \ref{LemCond}), yielding what is effectively an edge-disorder partition function. In the next subsection, we introduce a map $\mathcal{Q}$ that is used to decompose edge-disorder partition functions. \vspace{.2cm}

In the sequel, we refer exclusively to the case $b=s$. The dependence of  all previously defined expressions on the parameter $b\in \{2,3, \ldots\}$ will be suppressed as follows:
\begin{align*}
D^{b,b}_n\,\equiv\, D_n\,,\,\,\,\,\Gamma^{b,b}_n\,\equiv\, \Gamma_n\,,\,\,\,\, \beta_{n,r}^{(b)}\,\equiv \,\beta_{n,r}\,,\,\,\,\, M^{b,b}(x)\,\equiv\,M(x)\,,\,\,\,\, R_b(r)\equiv R(r)\,,\,\,\,\, \kappa_b\,\equiv\, \kappa\,,\,\,\,\, \eta_b\equiv \eta\,.
\end{align*}

\subsection{Hierarchical symmetry and the $\boldsymbol{\mathcal{Q}}$ map}

The recursive construction of the diamond hierarchical graphs outlined in Section~\ref{SecDHG} implies a  canonical one-to-one correspondence between the set $E_n$ of edges on the $n^{th}$-generation diamond graph $D_n$ and the $2n$-fold product set $(\{1,\ldots,b\}\times\{1,\ldots,b\}\big)^n$. When $n=1$, the edges in $E_1$ are labeled by ordered pairs $(i,j)$ wherein $i\in\{1,...,b\}$ enumerates the branches of $D_1$ and $j\in\{1,...,b\}$ enumerates the segments along the $i^{th}$ branch. The general case of the correspondence then follows from induction since there is a canonical bijection between $E_n$ and $E_1\times E_{n-1}$ arising directly from the embedding procedure used to define the diamond graphs; see the diagram in Figure~\ref{Fig2}. For $N<n$, the edge set $E_N$ is canonically bijective to a family of subgraphs $\{D_n^h\}_{h \in E_N}$ of $D_n$, each of which is isomorphic to $D_{n-N}$; see Figure~\ref{Fig3}. For $h\in E_N$, we use $V_n^h$ and $E_n^h$ respectively to denote the non-root vertex set and edge set of $D_n^h$. Under the natural identification of $V_N$ with a subset of $V_n$, the collection of sets $\{V_n^h\}_{h\in E_N}$ forms a partition of $V_n\backslash V_N$.

\begin{figure}[hbt!]
\centering
\includegraphics[width=130mm, height=80mm]{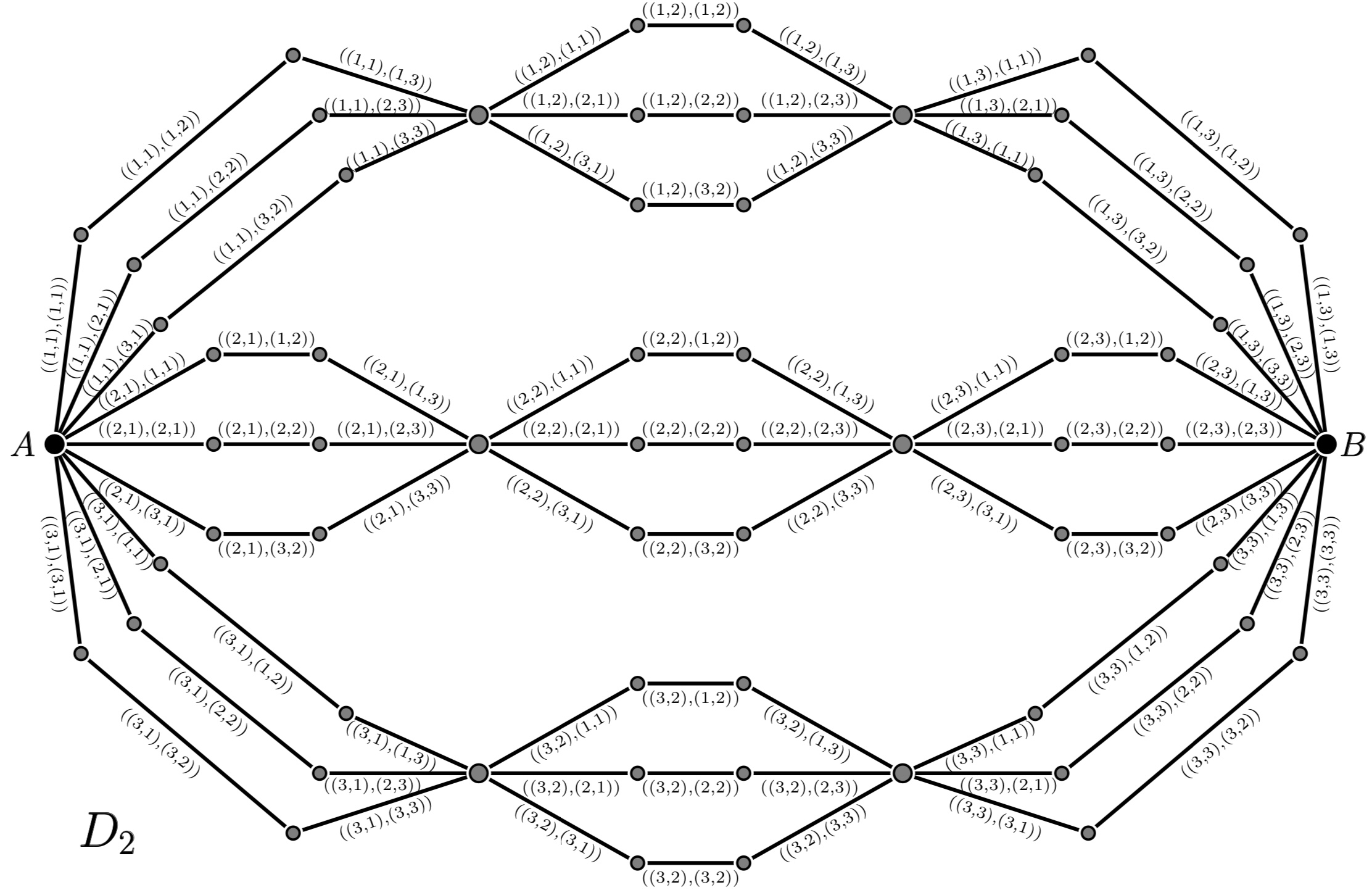}
\begin{minipage}{.81\textwidth}
\vspace{.2cm}
\caption{\label{Fig2} The above depicts the second-generation diamond graph in the case $b=s=3$. Each edge $h\in E_2$ is labeled by a nested tuple $\big((i_1,j_1),(i_2,j_2)\big)$ for some $i_1,j_1,i_2,j_2\in\{1,2,3\}$, where the ordered pair $(i_1,j_1)$ indicates the embedded subgraph isomorphic to $D_1$ in which $h$ lies and $(i_2,j_2)$ specifies the branch and segment for the position of $h$ within that subgraph.}
\end{minipage}
\end{figure}

\begin{figure}[hbt!]
\centering
\includegraphics[width=140mm, height=48mm]{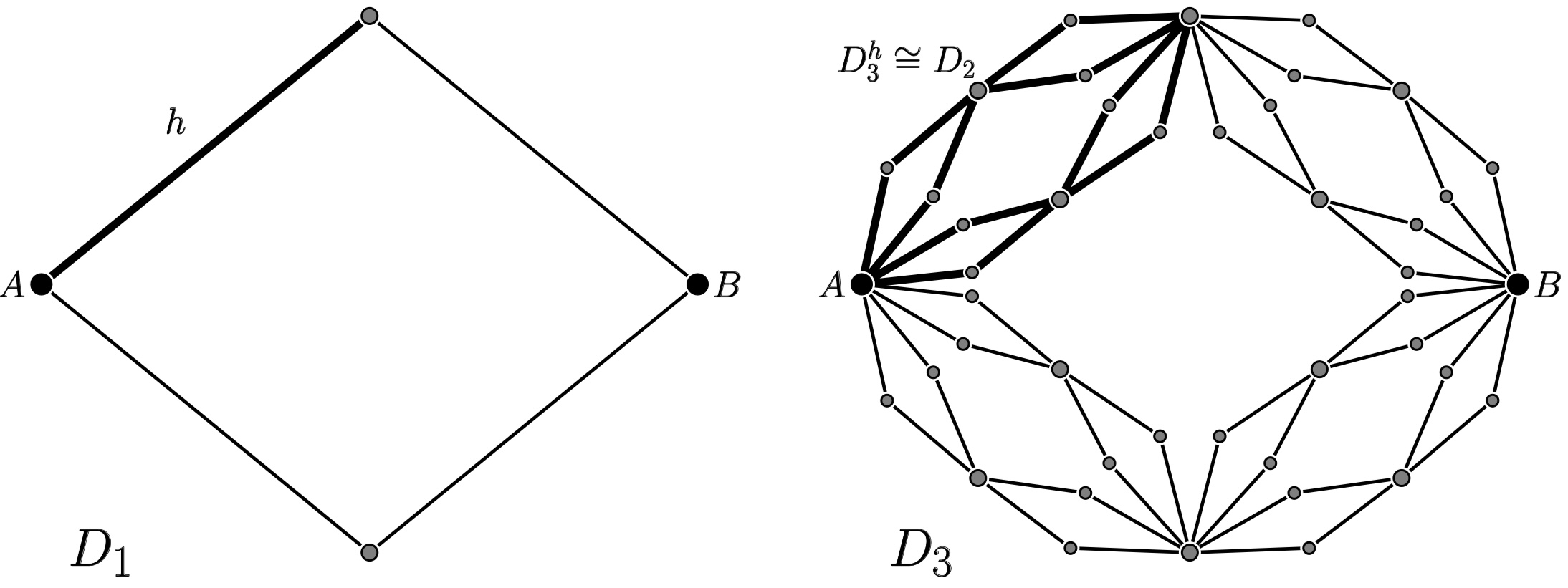} 
\begin{minipage}{.81\textwidth}
\vspace{.2cm}
\caption{\label{Fig3} For the case $b=s=2$, the edge $h\in E_1$ corresponds to an embedded copy, $D_3^h$, of the diamond graph $D_2$ within $D_3$.  More generally, for $N\leq n$, each $h\in E_N$ is associated with a subgraph $D_n^h \cong D_{n-N}$ of $D_n$.}
\end{minipage}
\end{figure}

 The following defines an operation $\mathcal{Q}$ that contracts arrays of real numbers indexed by $E_k$, which will be useful in the following subsection.

\begin{definition}\label{DefArrayMap}

For $k\in \mathbb{N}_0$, let $\{x_h\}_{h\in E_k}$ be an array of real numbers labeled by $ E_k$. Given $h\in  E_{k-1}$, let $h{\times} (i,j)$ for $i,j\in \{1,\ldots, b\}$ denote the element in $E_{k}$ corresponding to the $j^{th}$ segment along the $i^{th}$ branch of $D_k^h$, in other terms the embedded copy of  $D_1$ in $D_{k}$ identified with $h$. 
 We define $\mathcal{Q}$  as the  map that sends an array of real numbers $\{x_{h}\}_{h\in E_k}$ to the contracted array
\begin{align*}
\mathcal{Q}\{x_{h}\}_{h\in E_k} \,:=\,\{w_h\}_{h\in E_{k-1}  }\hspace{.5cm}\text{ for  }\hspace{.5cm} w_h \,:=\,\frac{1}{b}\sum_{1 \leq i \leq b}\Bigg( \prod_{1\leq j\leq b} \big(1+x_{h{\times}(i,j)}\big) \,-\,1\Bigg) \,. 
\end{align*}
 For $N\in \mathbb{N}_0$, $\mathcal{Q}^{N}$  refers to the $N$-fold composition of the $\mathcal{Q}$ map.\footnote{Note that the notation $\mathcal{Q}$ is ambiguous because it simultaneously denotes maps from $\R^{E_k}$ to $\R^{E_{k-1}}$ for each $k\in \mathbb{N}$.} 
\end{definition}

\subsection{Removing lower generation vertices from the partition function}

For $N\leq n$, let $\mathcal{F}_{n}^{N}$ be the $\sigma$-algebra generated by the family of random variables $\{\omega_a\}_{a \in V_n\backslash V_N}$. The lemma below, which we prove in Section~\ref{SecLemCond}, states that  the partition function $W_n^{\omega}(\beta_{n,r} )$ is not changed much (in the $L^2$ sense) by integrating out the disorder variables labeled by vertices of generation less than $\log n$ when $n$ is large. 
\begin{lemma}\label{LemCond} For fixed $r\in \R$, suppose that the sequence  $(\beta_{n,r})_{n\in \mathbb{N}}$ has the large-$n$ asymptotics in~(\ref{BetaForm2}).  When $N=\lfloor \log n\rfloor$, the $L^2$ distance between $W_{n}^{\omega}(\beta_{n,r})$ and $\widetilde{W}_{n}^{\omega}(\beta_{n,r}):=\mathbb{E}\big[W_{n}^{\omega}(\beta_{n,r})\,\big|\, \mathcal{F}_{n}^{N}  \big]$ vanishes as $n\uparrow \infty$. 
\end{lemma} 
Next, we will discuss how to express $\E\big[W_{n}^{\omega}(\beta)\,\big|\, \mathcal{F}_{n}^{N}\big]$ in terms of the $\mathcal{Q}$ map. The conditional expectation of the partition function $W_{n}^{\omega}(\beta)$ with respect to $\mathcal{F}_{n}^{N}$ has the form 
\begin{align} \label{Remove}
\E
\left[W_{n}^{\omega}(\beta) \, \Big| \, \mathcal{F}_{n}^{N}\right] \, = \,
\frac{1}{ |\Gamma_{n}|}\sum_{ p\in \Gamma_{n  }}\prod_{\substack{a \in  p \\ a\in V_n\backslash V_{N}  }} \textup{exp}\big\{\beta\omega_a -\lambda(\beta)\big\}\, .
\end{align}
In other words, the disorder variables corresponding to vertices in $V_N$ have been removed from the expression (\ref{PF}). For $h \in E_N$, let $\Gamma_n^h$ denote the path space on the corresponding embedded copy $D_n^h$ of $D_{n-N}$ within $D_n$. That is, if $\phi^{h}: V_{n-N} \rightarrow V_n^h$ is a graph isomorphism, then each $\mathbf{q }\in \Gamma_n^h$ is a function from $\{1,..., b^{n-N}-1\}$ to $V_n^h$ of the form $\mathbf{q }= \phi^{h} \circ \mathbf{p}$ for a unique $\mathbf{p} \in \Gamma_{n-N}$. Moreover, we have the following one-to-one correspondence between $\Gamma_n$ and a union of Cartesian products:
\begin{align} \label{Bijection}
    \Gamma_n \, \cong \, \bigcup_{q\in\Gamma_N} \prod_{h \pmb{\in} q} \Gamma_n^h \, ,
\end{align}
 in which $h \pmb{\in} q$ means that the edge $h\in E_N$ lies along the path $q$.\footnote{Recall that a Cartesian product $\prod_{\alpha \in A}S_{\alpha}$ is the set of functions from the set A to $\bigcup_{\alpha \in A} S_\alpha$ such that $f(\alpha) \in S_{\alpha}$ for each $\alpha \in A$.} The above states that each path $p \in \Gamma_n$ is determined by a generation-N coarse-graining $q \in \Gamma_N$ and a choice of a sub-path $\mathbf{q}^h\in \Gamma_n^h$ for each edge $h$ along $q$; see Figure~\ref{Fig4}. Given $p \in \Gamma_n$, let $q \in \Gamma_N$ be its coarse-graining and $\mathfrak{p}$ be the element in $\prod_{h \pmb{\in} q} \Gamma_n^h$ canonically corresponding to $p$. Notice that we can write the set of vertices on $p$ of generation higher than $N$ as the following disjoint union: 
 \begin{align} \label{PathPart}
     \big\{a\in p\, \big| \, a \in V_n \backslash V_N\big\}\, =\, \bigcup_{h \pmb{\in} q} \big\{a \in \mathfrak{p}(h)\big\}\,.
 \end{align}
\begin{figure}[hbt!]
\centering
\includegraphics[width=150mm, height=48mm]{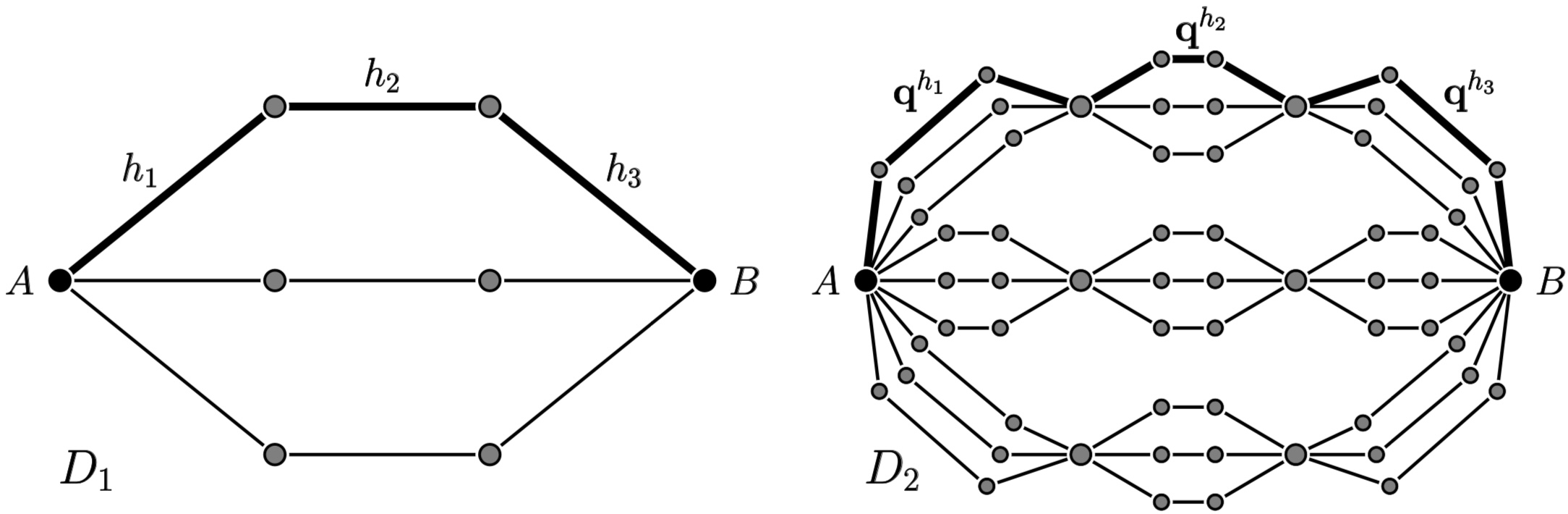}
\begin{minipage}{.81\textwidth}
\vspace{.2cm}
\caption{\label{Fig4} In the case $b=s=3$, the edges $h_1, h_2, h_3\in E_1$ comprise a  path $q\in \Gamma_1$ that is the generation-1 coarse-graining of a path $p\in\Gamma_2$ formed by the concatenated subpaths $\mathbf{q}^{h_i}\in \Gamma_2^{h_i}$  for  $i\in\{1,2,3\}$.  Thus, $p$ is identified with the tuple $\mathfrak{p}=\big(\mathbf{q}^{h_1},\mathbf{q}^{h_2},\mathbf{q}^{h_3}\big) $ in the one-to-one correspondence~(\ref{Bijection}).}
\end{minipage}
\end{figure}

\begin{definition}\label{DefMicroPart} We define the \textit{local partition function associated to} $h \in E_N$ as the random variable
$$ W_n^{h}(\beta)\,:=\,   \frac{1}{|\Gamma_{n-N}|} \sum_{\mathbf{q} \in  \Gamma_{n}^{h}  }    \prod_{ \substack{ a\in  \mathbf{q}  } } \textup{exp}\big\{\beta\omega_a -\lambda(\beta)\big\}  \, . $$
\end{definition}
Note that each $W_n^h(\beta)$ is equal in distribution to the partition function $W^\omega_{n-N}(\beta)$ since $\Gamma_n^h \cong \Gamma_{n-N} $. Furthermore, since the local vertex sets in the collection $\{V_n^h\}_{h \in E_N}$ are disjoint and $W_n^h(\beta)$ is a function of the array of random variables $\{\omega_a\}_{a\in V_n^h}$, the random variables in the family $\big\{W_n^h(\beta)\big\}_{h \in E_N}$ are i.i.d. The following proposition shows how the conditional expectation of the partition function in (\ref{Remove}) can be written in terms of the $\mathcal{Q}$ map and the family of local partition functions $\big\{W_{n}^{h}(\beta)\big\}_{h\in E_N}$. 

\begin{proposition}\label{PropReduce} Let $N,n\in \mathbb{N}_0$, and assume $N \leq  n$. For each $h \in E_N$, define $ X^{\omega}_h(\beta):= W_{n}^{h}(\beta)-1 $. The conditional expectation of $W_{n}^{\omega}(\beta)$ with respect to $\mathcal{F}_{n}^{N}$ can be written in the form
$$
\mathbb{E}\Big[W_{n}^{\omega}(\beta)\,\Big|\, \mathcal{F}_{n}^{N}  \Big]\,=\,1\,+\,\mathcal{Q}^{N}\big\{ X_h^{\omega}(\beta) \big\}_{h\in E_{N}  }\,.  $$ 
\end{proposition}

\begin{proof}Note that the one-to-one correspondence (\ref{Bijection}) implies that $|\Gamma_n| = |\Gamma_N|\cdot |\Gamma_{n-N}|^{b^N} $ because there are $b^N$ edges $h$ for each $q \in \Gamma_N$, and each subgraph $\Gamma_n^h$ is isomorphic to $\Gamma_{n-N}$. Thus, using~(\ref{PathPart}) we can write the right-hand side of (\ref{Remove}) as in the first equality below.
 \begin{align*} 
 \frac{1}{ |\Gamma_{n}|}\sum_{ p\in \Gamma_{n  }}\prod_{\substack{a \in  p \\ a\in V_n\backslash V_{N}  }} \textup{exp}\big\{\beta\omega_a -\lambda(\beta)\big\}\,=\,&\,\frac{1}{ |\Gamma_{N}|} \sum_{ q\in \Gamma_{N  } } \frac{1}{|\Gamma_{n-N}|^{b^N}}\sum_{\mathfrak{p} \in \prod_{h\pmb{\in} q} \Gamma_n^{ h }   }   \prod_{h\pmb{\in} q}  \prod_{a\in \mathfrak{p}(h) } \textup{exp}\big\{\beta\omega_a -\lambda(\beta)\big\} \\
\,=\,&\,\frac{1}{ |\Gamma_{N}|} \sum_{ q\in \Gamma_{N  } } \prod_{h \pmb{\in} q}\bigg( \frac{1}{|\Gamma_{n-N}|} \sum_{\mathbf{q}\in \Gamma_n^h  }    \prod_{a \in \mathbf{q} } \textup{exp}\big\{\beta\omega_a -\lambda(\beta)\big\} \bigg) \\
\,=\,&\,\frac{1}{ |\Gamma_{N}|} \sum_{ q\in \Gamma_{N } } \prod_{h \pmb{\in} q}W_{n}^h(\beta)\,=\,1\,+\,\mathcal{Q}^{N}\big\{ W_{n}^{h}(\beta)-1 \big\}_{h\in E_{N}  }\,
\end{align*}
The second equality above results from factoring, and the third equality holds since the expression in parentheses has the form of the local partition function $W_{n}^{h }(\beta)$ from Definition~\ref{DefMicroPart}. Finally, the last equality is equivalent to what was proved in~\cite[Proposition 5.5]{Clark2}. 
\end{proof}

\subsection{Variance and the $\boldsymbol{\mathcal{Q}}$ map}\label{SecVarQMap}

The previous subsection showed that for large $n$ and   $N=\lfloor\log n \rfloor$, the partition function $W_{n}^{\omega}(\beta_{n,r})$ is approximately equal (in $L^2$ norm) to an $N$-fold application of the $\mathcal{Q}$ map against an i.i.d.\ array of random variables indexed by $E_N$.  Note that if $\{X_h\}_{h\in E_k}$ is an i.i.d.\ family of centered random variables with variance $\mathcal{V}$, then the random variables in the array $\mathcal{Q}\{X_h\}_{h\in E_k}$ are centered with variance $M(\mathcal{V})=\frac{1}{b}\big[(1+\mathcal{V})^b-1\big]$.  If $M^k$ denotes the $k$-fold composition of the function $M:[0,\infty)\rightarrow [0,\infty)$, then the random variable $\mathcal{Q}^k\{X_h\}_{h\in E_k}$ has variance $M^k(\mathcal{V})$. The following proposition  from~\cite[Lemma 1.1]{Clark1} shows that  $M^k(\mathcal{V}_k)$ converges to a nontrivial limit when the sequence of positive numbers $(\mathcal{V}_k)_{k\in \mathbb{N}}$ vanishes in an appropriate way; see also~\cite[Appendix B]{Clark2} for a heuristic discussion of the consistency between properties (I) and (II) below.

\begin{proposition}\label{PropVar} For any $b\in \{2,3,\ldots\}$, there exists a unique continuously differentiable increasing function $R:\R\rightarrow (0,\infty)$  satisfying (I) and (II) below:
\begin{enumerate}[(I)]
\item Composing   $R(r)$ with the map $M$  translates the parameter $r$ by $1$:\, $ M\big(R(r)\big)\,=\, R(r+1) $.

\item As $r\uparrow \infty$, $R(r)$ diverges to $\infty$.  As $r\downarrow -\infty$, $R(r)$ has the vanishing asymptotics 
$$  R(r)\,=\,\frac{\kappa^2}{-r}\bigg(1 \,+\, \frac{ \eta\log(-r) }{ -r }\bigg)\,+\,\mathit{O}\bigg( \frac{\log^2(-r)}{|r|^3} \bigg)  \,. $$

\end{enumerate}
Moreover,   if for some $r\in \R$ the sequence of positive real numbers $(\mathcal{V}_{N,r})_{N\in \mathbb{N}} $ has the large-$N$ asymptotics 
\begin{align}\label{xAssump}
 \mathcal{V}_{N,r} \, = \, \frac{\kappa^2}{N} \bigg(1 \, +\,\frac{ \eta\log N}{N}\,+\,\frac{r}{N}\bigg)\,+\,\mathit{o}\Big(\frac{1}{N^2}  \Big) \,,
\end{align}
 then $ M^{N}(\mathcal{V}_{N,r}   )$ converges to $R(r) $ as $N\uparrow \infty$.
\end{proposition}

The above proposition and the  discussion preceding it suggest the possibility that if $\big(\{X_h^{(N)}\}_{h \in E_N}\big)_{N\in \mathbb{N}}$ is a  sequence of arrays of i.i.d.\ random variables with mean zero and variances $\mathcal{V}_{N,r}$ vanishing with the asymptotics in~(\ref{xAssump}), then the random variables $\mathcal{Q}^N\big\{X_h^{(N)}\big\}_{h \in E_N}$ converge in distribution as $N \uparrow \infty$. We explore this idea in the next subsection.  Recall that $X^{\omega}_h(\beta):=W^{h}_n(\beta)-1  $ and  that the random variables in the array $\big\{W^{h}_n(\beta)\big\}_{h \in E_N}$ are i.i.d.\ copies of $W_{n-N}^{\omega}(\beta)$. The connection between the inverse  temperature scaling $\beta_{n,r}>0$  in~(\ref{BetaForm2}) and the asymptotic form~(\ref{xAssump}) is given by the following lemma, which is our primary technical obstacle.
\begin{lemma}\label{LemMtilde} The variance of  $W_{n-\lfloor \log n\rfloor}^{\omega}(\beta_{n,r})$ has the large-$n$ asymptotics 
\begin{align}\label{Unflat}
\textup{Var}\Big( W_{n-\lfloor \log n\rfloor}^{\omega}(\beta_{n,r}) \Big)\,=\,\frac{\kappa^2}{\lfloor \log n \rfloor}\left(1 \,+\,\frac{ \eta\log\lfloor \log n\rfloor }{  \lfloor \log n\rfloor }\,+\, \frac{ r  }{ \lfloor \log n\rfloor }\right) \,+\, \mathit{o}\left(\frac{1}{\log^2 n}  \right)\,.
\end{align}
\end{lemma}
\vspace{.3cm}

 Our proof of Lemma~\ref{LemMtilde} refines a technique from the proof of~\cite[Lemma 5.16]{US} and is located in Section~\ref{SecVarAnal}. We will now summarize the structure of the  argument.  For $n\in \mathbb{N}$ and $r\in \R$, define the map $M_{n,r}(x) :=\frac{1}{b}\big[(1+x)^b \big(1+V_{n,r}\big)^{b-1}-1 \big]  $ on $[0,\infty)$, wherein  $V_{n,r}$ is the variance of the random variable $\textup{exp}\big\{\beta_{n,r}\omega-\lambda(\beta_{n,r}) \big\} $.  The large-$n$ asymptotic form~(\ref{BetaForm2}) for   $\beta_{n,r}$ implies that $V_{n,r}=\frac{\widehat{\kappa}^2}{n^2}\big( 1   \,+\,\frac{2\eta\log n }{n}\,+\, \frac{2r-2\varsigma}{n} \big)+\,\mathit{o}\big( \frac{1}{n^3}  \big)  $.
 Recall from~(\ref{RecEqVar}) that the sequence of variances
   $\big(\textup{Var}\big( W_{k}^{\omega}(\beta_{n,r}) \big)\big)_{k\in \N_0} $  satisfies the recursive relationship
   \begin{align}\label{VarRecur}
\textup{Var}\big( W_{k+1}^{\omega}(\beta_{n,r}) \big)\,=\, M_{n,r}\Big(\textup{Var}\big( W_{k}^{\omega}(\beta_{n,r}) \big)\Big) \,.
   \end{align}
   To perform our analysis, we transform  $\big( \textup{Var}\big( W_{k}^{\omega}(\beta_{n,r}) \big) \big)_{k\in \N_0}$ 
 to a sequence of values $\big( \mathbf{r}_k^{(n,r)} \big)_{k\in \N_0} $ in  the interval $[0,1)$ through the rule
$$ \mathbf{r}_k^{(n,r)}\,:= \,  \frac{2}{\pi}\tan^{-1}\bigg(\frac{2\mathbf{n}_{n,r}}{\pi \kappa^2}\textup{Var}\big( W_{k}^{\omega}(\beta_{n,r}) \big)\bigg) \,, $$
for $\mathbf{n}_{n,r}:=  \frac{\pi \kappa}{2}\big(\frac{b}{b-1}\big)^{1/2}V_{n,r}^{-1/2}$, which has the large-$n$ asymptotics
\begin{align}\label{nAsymp}
 \mathbf{n}_{n,r}\,=\,n\,-\,\eta\log n \,-\,r\, +\,\varsigma  \,+\,\mathit{o}(1) \,. 
 \end{align}
 The sequence $\big( \mathbf{r}_k^{(n,r)} \big)_{k\in \N_0}$ starts at  $\mathbf{r}_0^{(n,r)}=0$ and converges monotonically up to $1$, since the sequence $\big( \textup{Var}\big( W_{k}^{\omega}(\beta_{n,r}) \big) \big)_{k\in \N_0}$ is  increasing and  diverges to $\infty$ as a consequence of~(\ref{VarRecur}). 
 Moreover, the desired asymptotic form~(\ref{Unflat}) is equivalent to  
\begin{align}\label{EquivForm}
1\,-\,\mathbf{r}^{(n,r)}_{n-\lfloor\log n\rfloor}\,=\,\frac{\lfloor \log n\rfloor \,-\,\eta\log\log n\,-\,r}{n}\,+\,\mathit{o}\Big(\frac{1}{n}\Big) \,,
\end{align}
and our problem can thus be reframed in terms of $\big(\ \mathbf{r}_k^{(n,r)} \big)_{k\in \N_0} $. Since $\mathbf{r}^{(n,r)}_{0}=0$, the difference between $1$ and $\mathbf{r}^{(n,r)}_{n-\lfloor\log n\rfloor}$ can be rewritten in terms of the telescoping sum
 \begin{align}\label{Thissy}
 1\,-\,\mathbf{r}^{(n,r)}_{n-\lfloor \log n\rfloor}\,=\,\bigg(1\,-\,\frac{n-\lfloor \log n\rfloor    }{ \mathbf{n}_{n,r}  }\bigg)\,+\,\sum_{0 \leq k < n-\lfloor \log n\rfloor}\bigg(\mathbf{r}^{(n,r)}_{k}\,+\,\frac{1}{\mathbf{n}_{n,r}}\,-\,\mathbf{r}^{(n,r)}_{k+1}   \bigg) \,.
 \end{align}
A substantial portion of our analysis is directed towards converting the recursive relation~(\ref{VarRecur}) into an approximate form for $\mathbf{r}^{(n,r)}_{k}+\frac{1}{\mathbf{n}_{n,r}}-\mathbf{r}^{(n,r)}_{k+1}$ that telescopes within~(\ref{Thissy}), this being
\begin{align*}
\mathbf{r}^{(n,r)}_{k}&\, +\,\frac{1}{\mathbf{n}_{n,r}}\,-\,\mathbf{r}^{(n,r)}_{k+1}\ \\  \,\approx \,&\, -\frac{\eta}{n}\bigg[ \log\Big( \cos\Big(\frac{\pi}{2}\mathbf{r}^{(n,r)}_{k+1}\Big) \Big) \,-\,\log\Big( \cos\Big(\frac{\pi}{2}\mathbf{r}^{(n,r)}_{k}\Big)  \Big)       \bigg]\, 
-\, \frac{2\eta}{n}\bigg[ \sin^2\Big(\frac{\pi}{2}\mathbf{r}^{(n,r)}_{k+1}\Big) \,-\, \sin^2\Big(\frac{\pi}{2}\mathbf{r}^{(n,r)}_{k}\Big)      \bigg]   \,.      
\end{align*}
 After applying the above approximation,~(\ref{Thissy}) collapses to
\begin{align*}
 1\,-\,\mathbf{r}^{(n,r)}_{n-\lfloor \log n\rfloor}\, \approx \, \bigg(1\,-\,\frac{n-\lfloor \log n\rfloor    }{ \mathbf{n}_{n,r}  }\bigg)\, -\,\frac{\eta}{n} \log\Big( \cos\Big(\frac{\pi}{2}\mathbf{r}^{(n,r)}_{n-\lfloor \log n\rfloor}\Big) \Big) \,-\, \frac{2\eta}{n} \sin^2\Big(\frac{\pi}{2}\mathbf{r}^{(n,r)}_{n-\lfloor \log n\rfloor}\Big) \, 
\end{align*}
because  $\sin^2\big(\frac{\pi}{2}\mathbf{r}^{(n,r)}_{0}\big)=0  $ and $\log\big( \cos\big(\frac{\pi}{2}\mathbf{r}^{(n,r)}_{0}\big)  \big)=0$. Using~(\ref{nAsymp}) and the approximations $\cos\big(\frac{\pi}{2}x\big)\approx \frac{\pi}{2}(1-x)$ and $ \sin^2\big(\frac{\pi}{2}x\big)\approx 1$  for $x\in (0,1)$ close to $1$, we have
 \begin{align*}
 1\,-\,\mathbf{r}^{(n,r)}_{n-\lfloor \log n\rfloor}\, \approx\,&\, \frac{ \lfloor \log n\rfloor \,-\,\eta \log n\,-\,r\,+\,\varsigma    }{ n } \, -\,\frac{\eta}{n} \log\Big(\frac{\pi}{2} \Big( 1-\mathbf{r}^{(n,r)}_{n-\lfloor \log n\rfloor}\Big) \Big) \,-\, \frac{2\eta}{n}  \,. 
\end{align*}
At last, rearranging the above and using that $\varsigma=(\log\frac{\pi}{2}+2)\eta$, 
 \begin{align*}
  1\,-\,\mathbf{r}^{(n,r)}_{n-\lfloor \log n\rfloor}\, \approx\,&\,  \frac{\lfloor \log n\rfloor  \,  -\,\eta \log \log n\,-\,r  }{ n } \, -\,\frac{\eta}{n} \log\Bigg( \frac{ n\big(1-\mathbf{r}^{(n,r)}_{n-\lfloor \log n\rfloor} \big)}{ \log n }\Bigg)  \, .
 \end{align*}
  When rigorously formulated, this approximation for $1-\mathbf{r}^{(n,r)}_{n-\lfloor \log n\rfloor}$ can be plugged back into itself to complete the derivation of~(\ref{EquivForm}), the second term on the right side being $\mathit{o}\big(\frac{1}{n}\big)$.

\subsection{A limit theorem concerning the $\boldsymbol{\mathcal{Q}}$ map}

Before continuing to the proof Theorem~\ref{ThmMain}, we state a technical lemma and recall~\cite[Theorem 5.16]{Clark2} and~\cite[Theorem 5.22]{Clark2}, restated in Theorems~\ref{ThmExist} and~\ref{ThmUnique}, respectively.

\begin{definition}\label{DefRegular}
A sequence of edge-labeled arrays of random variables $\big(\{ X_h^{(N)} \}_{h\in  E_{N} }\big)_{N\in \mathbb{N}} $  taking values in $[-1,\infty)$ is said to be \textit{regular with parameter $r\in \R$} when the conditions (I)--(III) below hold.
\begin{enumerate}[(I)]

\item For each $N\in \mathbb{N}$, the random variables in the array $\big\{ X_h^{(N)} \big\}_{h\in  E_{N} } $ are centered and i.i.d.

\item The variance of the random variables in the array $\big\{ X_h^{(N)} \big\}_{h\in  E_{N} } $ has the large-$N$ asymptotics
\begin{align*}
\textup{Var}\big(  X_h^{(N)}  \big)\,=\,\frac{\kappa^2}{N} \bigg(1 \, +\, \frac{ \eta\log N}{N}\ +\,\frac{r}{N}\bigg)\,+\,\mathit{o}\Big(\frac{1}{N^2}  \Big)\,.
\end{align*}

\item For each $m\in \{4,6,\ldots\}$, the $m^{th}$ moment of the random variables in the array $\big\{ X_h^{(N)} \big\}_{h\in  E_{N} } $  vanishes as $N\uparrow \infty$.

\end{enumerate}
Moreover, $ \big(\{ X_h^{(N)} \}_{h\in  E_{N} }\big)_{N\in \mathbb{N}}$  is called \textit{minimally regular} if (I) and (II) hold, but (III) is only assumed for $m=4$.
\end{definition} 
 
To apply Theorem~\ref{ThmUnique} below in the proof of Theorem~\ref{ThmMain}, we observe that when $\lfloor \log n \rfloor $ is identified with $N$, the array of centered local partition functions $\big\{ X^{\omega}_h(\beta_{n,r})\big\}_{h\in E_{\lfloor \log n \rfloor }} $  satisfies the conditions of Definition~\ref{DefRegular}.\footnote{More precisely, if we choose any sequence of natural numbers $(a_N)_{N\in \mathbb{N}} $   such that $\lfloor \log  a_N \rfloor =N  $, then $\big\{ X^{\omega}_h(\beta_{a_N,r})\big\}_{h\in E_{N}} $ is a regular sequence with parameter $r$ when indexed by $N$.}     Condition (II) holds as a consequence of Lemma~\ref{LemMtilde}  since each $X^{\omega}_h(\beta_{n,r})$ is equal distribution to $ W_{n-\lfloor \log n\rfloor}^{\omega}(\beta_{n,r})$.  Moreover, the lemma below, which we prove in Section~\ref{SecLemmaHM}, verifies condition (III).
\begin{lemma}\label{LemmaHM} For each $m\in \mathbb{N}$, the $m^{th}$ centered moment of  $W_{n-\lfloor \log n\rfloor}^{\omega}(\beta_{n,r})$ vanishes as $n\uparrow \infty$. 
\end{lemma}

Let the  function $R:\R\rightarrow (0,\infty)$ be defined as in Proposition~\ref{PropVar}. The distribution $\mathbf{X}_r$ in the statement of the theorem below is related to the limiting law $\mathbf{W}_r$ from Theorem~\ref{ThmMain} through $\mathbf{W}_{r}=1+\mathbf{X}_r$. 
\begin{theorem}\label{ThmExist} There exists a unique family of  probability measures $\{\mathbf{X}_r\}_{r\in\R}$ supported on 
$[-1,\infty)$ having properties (I)--(III) below for each $r\in \R$. 
\begin{enumerate}[(I)]

\item The distribution $\mathbf{X}_r$ has mean zero and variance $R(r)$.

\item The distribution $\mathbf{X}_r$ has finite fourth moment that vanishes as $r\downarrow -\infty$.

\item If $\big\{X_r^{(i,j)}\big\}_{1\leq i,j\leq b}$ is a family of independent random variables with distribution $\mathbf{X}_r$, then \vspace{-.2cm}
$$\mathbf{X}_{r+1} \, \stackrel{d}{=} \, \frac{1}{b} \sum_{1 \leq i \leq b}\Bigg(\prod_{1\leq j \leq b} \left(1+X_r^{(i,j)}\right)\,-\,1\Bigg)\,. \vspace{-.2cm} $$ 
\end{enumerate}
Moreover, the integer moments of the distribution  $\mathbf{X}_r$ are finite for each $r\in \R$.
\end{theorem}

The limiting distribution in the following theorem is that from Theorem~\ref{ThmExist}.

\begin{theorem}\label{ThmUnique} Let $\big(\{ X_h^{(N)} \}_{h\in  E_{N} }\big)_{N\in \mathbb{N}} $ be a minimally regular  sequence  of arrays of random variables with parameter  $r\in \R$.  Then the sequence of random variables $\big(\mathcal{Q}^N \{ X_h^{(N)} \}_{h\in  E_{N} } \big)_{N \in \N}$ converges in law to $\mathbf{X}_r$ as $N\uparrow \infty$.
\end{theorem}

\subsection{Proof of Theorem~\ref{ThmMain} }\label{SecSiteDisorder}

\begin{proof}
Let the i.i.d.\ array of random variables  $\big\{X_{h}^{\omega}(\beta_{n,r})\big\}_{h\in E_{\lfloor \log n\rfloor  }} $ be defined as in Proposition~\ref{PropReduce}. By Lemma~\ref{LemCond}, the $L^2$ distance between the generation-$n$  partition function $W_{n}^{\omega}(\beta_{n,r})$ and the random variable
$$\widetilde{W}_{n}^{\omega}(\beta_{n,r})\,:=\, \mathbb{E}\Big[ W_{n}^{\omega}(\beta_{n,r})  \,\Big|\,\mathcal{F}_{n}^{\lfloor \log n \rfloor}  \Big]  \,=\,1\,+\,\mathcal{Q}^{\lfloor \log n\rfloor}\big\{ X_h^{\omega}(\beta_{n,r}) \big\}_{h\in E_{\lfloor \log n\rfloor}  }  $$ 
vanishes with large $n$, where the second equality holds by Proposition~\ref{PropReduce}.  Thus, it suffices to prove that the random variables $\mathcal{Q}^{\lfloor \log n\rfloor}\big\{ X_h^{\omega}(\beta_{n,r}) \big\}_{h\in E_{\lfloor \log n\rfloor}  } $ converge in distribution to $\mathbf{X}_r =\mathbf{W}_{r}-1$  as $n \uparrow \infty$. Observe that the statements (I)--(III) below hold.
\begin{enumerate}[(I)]
\item The random variables in the array $\big\{ X_h^{\omega}(\beta_{n,r}) \big\}_{h\in E_{\lfloor \log n\rfloor}}$ are independent copies of $W_{n-\lfloor \log n\rfloor }^{\omega}(\beta_{n,r})-1$ as a consequence of the discussion preceding Proposition~\ref{PropReduce}.

\item By Lemma~\ref{LemMtilde}, the variance of the random variable $W_{n-\lfloor \log n\rfloor }^{\omega}(\beta_{n,r})$ has the large-$n$ asymptotics 
$$ \textup{Var}\Big( W_{n-\lfloor \log n\rfloor }^{\omega}(\beta_{n,r}) \Big)\,=\,\frac{\kappa^2}{\lfloor \log n \rfloor}\left(1+\,\frac{\eta \log \lfloor \log n\rfloor    }{ \lfloor \log n\rfloor } \,+\,\frac{ r  }{ \lfloor \log n\rfloor} \right) \,+\,\mathit{o}\left( \frac{1}{ \log^2 n  }  \right)  \,.    $$

\item For each $m\in \{4, 6,\ldots\}$,  the $m^{th}$ centered moment of $W_{n-\lfloor \log n\rfloor }^{\omega}(\beta_{n,r})$ vanishes as $n\uparrow \infty$ by Lemma~\ref{LemmaHM}.

\end{enumerate}
Statements (I)--(III) imply that the sequence in $n\in \mathbb{N}$ of edge-labeled arrays $\big\{ X_{h}^{\omega}(\beta_{n,r})\big\}_{h\in E_{\lfloor \log n\rfloor}}$
satisfies the  conditions (I)--(III) in Definition~\ref{DefRegular} with $N=\lfloor \log n\rfloor $.  Thus, by Theorem~\ref{ThmUnique}, the random variables $\mathbf{X}^{(n)}:= \mathcal{Q}^{\lfloor \log n\rfloor}\big\{ X_h^{\omega}(\beta_{n,r}) \big\}_{h\in E_{\lfloor \log n\rfloor}  } $ converge in distribution to $\mathbf{X}_r$ with large $n$.  Therefore, $W_{n}^{\omega}(\beta_{n,r})$ converges in distribution to $\mathbf{W}_r$ as $n\uparrow \infty$.
\end{proof}

\section{Proofs of the three lemmas}\label{SecThreeLemmas}
In this section, we provide the proofs of Lemmas~\ref{LemCond},~\ref{LemMtilde}, and~\ref{LemmaHM}. We begin with Lemma~\ref{LemMtilde},  because its application is used to show the other two lemmas. As previously mentioned, the analysis in the proof of Lemma~\ref{LemMtilde} improves on that of~\cite[Theorem 2.5]{US}.
\subsection{Proof of Lemma~\ref{LemMtilde}}\label{SecVarAnal}

For $k\in \mathbb{N}_{0}$ and $\beta>0$, recall that  $\varrho_{k}(\beta)$ denotes the variance of  $W^{\omega}_{k}(\beta)$ and that the sequence of variances  $\big({\varrho}_{k}(\beta)\big)_{k\in \mathbb{N}_0}$ satisfies the recursive equation~(\ref{RecEqVar}), where for $b=s$ the map $M_{V}:[0,\infty)\rightarrow [0,\infty)$ is defined by
\begin{align}\label{ReDefMHat}
   M_{V}(x) \,:=\,\frac{1}{b}\Big[ (1+x)^b\big(1+ V \big)^{b-1}    \,-\,1   \Big]  \hspace{.5cm} \text{ for } \hspace{.5cm} V\,:=\,\textup{Var}\big( \textup{exp}\big\{ \beta \omega -\lambda(\beta)\big\}  \big) \,.
\end{align}
The inverse temperature scaling~(\ref{BetaForm2}) results in the following  variance asymptotics as $n\uparrow \infty$:
\begin{align*}
 V_{n,r}\,:=\, \textup{Var}\big(  \textup{exp}\big\{\beta_{n,r}\omega-\lambda(\beta_{n,r})\big\}\big)\,=\, \frac{\widehat{\kappa}^2}{n^2}\bigg( 1   \,+\,\frac{2\eta\log n }{n}\,+\, \frac{2(r-\varsigma)}{n} \bigg)+\,\mathit{o}\Big( \frac{1}{n^3}  \Big)\,, 
 \end{align*}
 where, recall, $\widehat{\kappa}^2:=\frac{ \pi^2 b }{2(b-1)^2 }$, $\kappa^2:=\frac{ 2 }{b-1 }$, $\eta :=\frac{b+1}{3(b-1)}$, and $\varsigma := (\log\frac{\pi}{2}+2)\eta $.
It will be convenient to write $ V_{n,r}$ in the form $ V_{n,r}= \frac{ \widehat{\kappa}^2}{\mathbf{n}_{n,r}^2} = \frac{b\pi^2 \kappa^2}{(b-1)4\mathbf{n}_{n,r}^2}$ for  $\mathbf{n}_{n,r}:=  \frac{\pi \kappa}{2}\big(\frac{b}{b-1}\big)^{1/2}V_{n,r}^{-1/2}$, which has the large-$n$ asymptotics
\begin{align}\label{NDef}
 \mathbf{n}_{n,r}\,=\,n\,-\,\eta\log n \,-\,r\, +\,\varsigma  \,+\,\mathit{o}(1) \,. 
 \end{align}

\vspace{.1cm}

\begin{proof}[Proof of Lemma~\ref{LemMtilde}] We separate the proof into parts (A)--(H).\vspace{.15cm}

\noindent \textit{(A) An approximation for the variance map:} Since the variance $\varrho_k(\beta_{n,r}  )$ of $W_{k}^{\omega}(\beta_{n,r})$ satisfies the recursive equation~(\ref{RecEqVar}) in $k\in \mathbb{N}_{0}$, we have that
\begin{align}
\textup{Var}\Big( W_{n-\lfloor \log n\rfloor}^{\omega}(\beta_{n,r}) \Big)\,=\, M_{n,r}^{n-\lfloor\log n\rfloor }(0)\,,\nonumber 
\end{align}
where $M_{n,r}\equiv M_{V_{n,r}}$.  Let $\widetilde{M}_{n,r}:[0,\infty)\rightarrow [0,\infty)$ be defined through the following approximation of the expression for $M_{n,r}(x)$ in~(\ref{ReDefMHat}) around $(x,V_{n,r})=(0,0)$ that is third-order in $x$ and first-order in $V_{n,r}$:
\begin{align*}
        \widetilde{M}_{n,r}(x)\,:=\,&\,x\,+\,\frac{b-1}{2}x^2\,+\,\frac{(b-1)(b-2)}{6}x^3  \, +\,\frac{b-1}{b}V_{n,r}\big(1+bx\big)\\     \,=\,&\, x\,+\,\frac{x^2}{\kappa^2}\,+\,(1-\eta)\frac{x^3}{\kappa^4}\,+\,\frac{\pi^2 \kappa^2  }{4 \mathbf{n}_{n,r}^2 }\big(1+bx\big)  \,, \nonumber  
\end{align*}
where  $\mathbf{n}_{n,r}  $ is defined above~(\ref{NDef}). Note that the definition of $\widetilde{M}_{n,r}$ retains the lowest-order cross term $(b-1)V_{n,r} x$.  Define $\mathscr{E}(x, \mathbf{n}_{n,r}  ):= M_{n,r}(x)-\widetilde{M}_{n,r}(x) $, in other words,  the error of the approximation of $M_{n,r}$ by $\widetilde{M}_{n,r}$.   When $x\geq 0$, the error is nonnegative and has the following bound for some $\mathbf{c}>0$ and all $n\in \mathbb{N}$ and  $0\leq x \leq 1$:
\begin{align}\label{ErrorTerm}
\mathscr{E}(x, \mathbf{n}_{n,r}  )\,\leq\,\mathbf{c}\big( x^4 \,+\,\mathbf{n}_{n,r}^{-4} \big)\,.
\end{align}
The above inequality can be shown by expanding the expression~(\ref{ReDefMHat}) in $x$ and $V$, and then applying Young's inequality to the cross terms $x^iV_{n,r}^j$ with $i+j>2$, of which the lowest-order  is $x^2V_{n,r}\propto \frac{x^2}{\mathbf{n}_{n,r}^2}$. \vspace{.2cm}

\noindent \textit{(B) Transforming the variables:}
For $r\in \R$ and $n\in \mathbb{N}$, define the sequence $\big( \mathbf{r}_k^{(n,r)} \big)_{k\in \mathbb{N}_0}$ of  elements  in $[0,1)$ as 
\begin{align}\label{FormR}
\mathbf{r}_k^{(n,r)}\,:=\, \frac{2}{\pi}\tan^{-1}\bigg(\frac{2\mathbf{n}_{n,r}}{\pi \kappa^2}M_{n,r}^k(0)\bigg)\,,\hspace{.3cm} \text{so that we have  }\hspace{.3cm} \frac{\pi \kappa^2}{2}\tan\Big( \frac{\pi}{2}\mathbf{r}_k^{(n,r)}\Big)\,=\,\mathbf{n}_{n,r} M_{n,r}^k(0)\,.
\end{align}
Note that $\mathbf{r}_0^{(n,r)}=0$ since $M_{n,r}^0(0)=0$.  For notational neatness, we identify $\mathbf{r}_k^{(n,r)}\equiv \mathbf{r}_k$, i.e., suppress the dependence on the superscript variables.  The sequence $( \mathbf{r}_k )_{k\in \mathbb{N}_0}$ converges monotonically up to $1$ as $k\uparrow \infty$, and it will suffice for us to show that 
\begin{align}\label{Flattened}
1\,-\,\mathbf{r}_{n-\lfloor\log n\rfloor}\,=\,\frac{\lfloor \log n\rfloor -\eta\log\log n\,-\,r}{n}\,+\,\mathit{o}\Big(\frac{1}{n}\Big)\,.
\end{align}
To see the equivalence between~(\ref{Flattened}) and~(\ref{Unflat}), note that for large $n$---and thus for small $1-\mathbf{r}_{n-\lfloor\log n \rfloor }$ values---we  get the second equality below  through second-order Taylor expansions of $f_1(x)=\sin\big( \frac{\pi}{2}x \big)$ and $f_2(x)=\cos\big( \frac{\pi}{2}x \big)$ around $x=1$:
\begin{align*}
 \mathbf{n}_{n,r} M_{n,r}^{n-\lfloor \log n\rfloor  }(0)\,=\,\frac{\pi \kappa^2}{2}\tan\Big( \frac{\pi}{2}\mathbf{r}_{n-\lfloor\log n\rfloor}\Big) \,=\,  \frac{\kappa^2}{1-\mathbf{r}_{n-\lfloor\log n \rfloor }}\,+\,\mathit{O}\big( 1-\mathbf{r}_{n-\lfloor\log n\rfloor } \big) \,.
\end{align*} 
Finally, recall from~(\ref{NDef}) that $ \mathbf{n}_{n,r}=n+\mathit{O}( \log n  ) $ for large $n$.  Thus we only need to prove~(\ref{Flattened}). \vspace{.3cm}

\noindent \textit{(C) Rewriting the increments of $( \mathbf{r}_k )_{k\in \mathbb{N}_0}$ using Taylor's theorem:} By writing $M_{n,r}^{k+1}(0)=M_{n,r}\big(M_{n,r}^{k}(0)\big)  $ and splitting $M_{n,r}$ into a sum of  $\widetilde{M}_{n,r}$ and the error term $\mathscr{E}$, we get the equality
\begin{align*}
\mathbf{n}_{n,r} M_{n,r}^{k+1}(0)\,=\,\, &\mathbf{n}_{n,r} M_{n,r}^{k}(0) \,+\,\underbracket{\frac{1}{\kappa^2\mathbf{n}_{n,r} }\Big( \mathbf{n}_{n,r}M_{n,r}^{k}(0) \Big)^2}\,+\,\frac{1-\eta}{\kappa^4\mathbf{n}_{n,r}^2 }\Big(\mathbf{n}_{n,r} M_{n,r}^{k}(0) \Big)^3\,+\,\underbracket{\frac{\pi^2 \kappa^2  }{4 \mathbf{n}_{n,r} }}  \\ & \,+\,\frac{b\pi^2 \kappa^2  }{4 \mathbf{n}_{n,r}^2 }\Big(\mathbf{n}_{n,r} M_{n,r}^{k}(0)\Big)  \,+\, \mathbf{n}_{n,r}\mathscr{E}\left(M_{n,r}^{k}(0), \mathbf{n}_{n,r} \right)  \,.
\end{align*}
With~(\ref{FormR}), we can rewrite the equation  above using the variables $\mathbf{r}_{k}$ and $\mathbf{r}_{k+1}$ as below, where the under-bracketed expressions have combined to form the $\sec^2$ term.
\begin{align*}
\mathbf{r}_{k+1}\,=\,\,&\frac{2}{\pi}\tan^{-1}\Bigg(\tan\Big( \frac{\pi}{2}\mathbf{r}_k   \Big)   \,+\,\frac{\pi}{2\mathbf{n}_{n,r}}\sec^2\Big( \frac{\pi}{2}\mathbf{r}_k   \Big)\nonumber \\  &\text{}\hspace{.2cm}\,+\, \underbrace{\frac{\pi^2}{4\mathbf{n}_{n,r}^2}(1-\eta)\tan^3\Big( \frac{\pi}{2}\mathbf{r}_k   \Big)\,+\,\frac{b\pi^2\kappa^2}{4\mathbf{n}_{n,r}^2}\tan\Big( \frac{\pi}{2}\mathbf{r}_k   \Big)\,+\,\frac{2}{\pi \kappa^2} \mathbf{n}_{n,r} \mathscr{E}\bigg( \frac{\pi\kappa^2}{2 \mathbf{n}_{n,r}}\tan\Big(\frac{\pi}{2} \mathbf{r}_k \Big) , \mathbf{n}_{n,r}  \bigg)}_{\mathbf{(I)}}  \Bigg)\,.
\end{align*}
If $\mathbf{r}_{k}< 1-\frac{1}{\mathbf{n}_{n,r}}$,  Taylor's theorem applied to  the function $g(x)=\tan\big(\frac{\pi}{2}x\big) $ around the point  $x=\mathbf{r}_{k}$ with second-order error implies that there is an $\mathbf{r}_k^* \in [\mathbf{r}_{k},\mathbf{r}_{k}+\frac{1}{\mathbf{n}_{n,r}})$ for which
\begin{align}\label{TanTwo}
\mathbf{r}_{k}+\frac{1}{\mathbf{n}_{n,r}}\,=\,\frac{2}{\pi}\tan^{-1}\bigg(\tan\Big( \frac{\pi}{2}\mathbf{r}_k  \Big)   \,+\,\frac{\pi}{2\mathbf{n}_{n,r}}\sec^2\Big( \frac{\pi}{2}\mathbf{r}_k  \Big) \,+\, \underbrace{\frac{\pi^2}{4\mathbf{n}_{n,r}^2}\tan\Big( \frac{\pi}{2}\mathbf{r}_k^*   \Big) \sec^2\Big( \frac{\pi}{2}\mathbf{r}_k^*   \Big)}_{\mathbf{(II)}} \bigg)\,.
\end{align}
Define $\Delta_k$ as the difference between the expressions  $(\mathbf{II})$ and $(\mathbf{I})$, which can be written as
\begin{align*}
\Delta_k\,:=\,\,&\frac{\pi^2}{4\mathbf{n}_{n,r}^2}\bigg(\tan\Big( \frac{\pi}{2}\mathbf{r}_k^*   \Big) \sec^2\Big( \frac{\pi}{2}\mathbf{r}_k^*   \Big) \,-\,(1-\eta)\tan^3\Big( \frac{\pi}{2}\mathbf{r}_k  \Big) \bigg)\,-\,\frac{b\pi^2\kappa^2}{4\mathbf{n}_{n,r}^2}\tan\Big( \frac{\pi}{2}\mathbf{r}_k   \Big)\nonumber\\ &-\,\,\frac{2}{\pi \kappa^2} \mathbf{n}_{n,r} \mathscr{E}\bigg( \frac{\pi\kappa^2}{2 \mathbf{n}_{n,r}}\tan\Big(\frac{\pi}{2} \mathbf{r}_k \Big) , \mathbf{n}_{n,r}  \bigg)\,.
\end{align*}
By applying Taylor's theorem to the function $h(x)=\frac{2}{\pi}\tan^{-1}(x)$ around the point $x=\tan\big(\frac{\pi}{2}\mathbf{r}_{k+1}  \big)$, we have an $\mathbf{r}_k^{**}$ between $\mathbf{r}_{k+1}$ and $\mathbf{r}_k+\frac{1}{\mathbf{n}_{n,r}}$ such that
\begin{align}
\mathbf{r}_k\,+\,\frac{1}{\mathbf{n}_{n,r}}\,=\,\,& \mathbf{r}_{k+1}\,+\,\frac{2}{\pi}\Delta_k\frac{1}{ 1+\tan^2\big(\frac{\pi}{2}\mathbf{r}_{k+1}   \big)    }\,-\,\frac{2}{\pi}\Delta_k^2\frac{\tan\big(\frac{\pi}{2}\mathbf{r}_{k}^{**}   \big)}{\big( 1+\tan^2\big(\frac{\pi}{2}\mathbf{r}_{k}^{**}   \big)\big)^2    } \nonumber  \\
\,=\,\,&\mathbf{r}_{k+1}\,+\,\frac{2}{\pi}\Delta_k \cos^2\Big(\frac{\pi}{2}\mathbf{r}_{k+1}   \Big)    \,-\,\frac{2}{\pi}\Delta_k^2\sin\Big( \frac{\pi}{2}\mathbf{r}_{k}^{**}  \Big)\cos^3\Big( \frac{\pi}{2}\mathbf{r}_{k}^{**}  \Big)\,.\label{Tintin}
\end{align}

Finally, we can use that $\tan^2(x)=\sec^2(x)-1 $ to write  $\Delta_k=\Delta_k'+\Delta_k''+\Delta_k'''$, where
\begin{align*}
\Delta_k'\,:=\,\, &\frac{\pi^2}{4\mathbf{n}_{n,r}^2}\bigg(\tan\Big( \frac{\pi}{2}\mathbf{r}_k^*   \Big) \sec^2\Big( \frac{\pi}{2}\mathbf{r}_k^*   \Big)\,-\,(1-\eta)\tan\Big( \frac{\pi}{2}\mathbf{r}_k  \Big)\sec^2\Big( \frac{\pi}{2}\mathbf{r}_k  \Big) \bigg) \,,  \\
\Delta_k'' 
\,:=\,\, &\frac{\pi^2}{4\mathbf{n}_{n,r}^2} \big(1-\eta -b\kappa^2\big) \tan\Big( \frac{\pi}{2}\mathbf{r}_k  \Big)\,=\,-\frac{\pi^2 \eta}{\mathbf{n}_{n,r}^2  } \tan\Big( \frac{\pi}{2}\mathbf{r}_k  \Big) \,, \,\,\,\text{and}\\
\Delta_k''' 
\,:=\, &-\frac{2}{\pi \kappa^2} \mathbf{n}_{n,r} \mathscr{E}\bigg( \frac{\pi\kappa^2}{2 \mathbf{n}_{n,r}}\tan\Big(\frac{\pi}{2} \mathbf{r}_k \Big) , \mathbf{n}_{n,r}  \bigg) \,.
\end{align*}

\vspace{.1cm}

\noindent \textit{(D) Bounds for the various terms in~(\ref{Tintin}):}  The inequalities below hold for some $C>0$ and all $k\in \mathbb{N}_0$ and $n\in \mathbb{N}$ such that $1-\mathbf{r}_{k}\geq \frac{\log n}{2 n } > \frac{1}{ \mathbf{n}_{n,r}}  $.\footnote{The lower bound of $1-\mathbf{r}_{k}$ by $\frac{1}{ \mathbf{n}_{n,r} }$ ensures that $\mathbf{r}_{k}^*$ is well-defined by~(\ref{TanTwo}). When $n$ is sufficiently large, $\frac{\log n}{2 n } > \frac{1} {\mathbf{n}_{n,r}} $ holds as a trivial consequence of~(\ref{NDef}).}  
\begin{enumerate}[(i)]
 \item $  0\,\leq \,\big|\Delta_k'''\big| \,\leq \, \frac{C}{ n^3(1-\mathbf{r}_k)^4 }  $

\item $|\Delta_k| \,\leq \,\frac{ C  }{ n^2(1-\mathbf{r}_k)^3  }$

\item $\big| \mathbf{r}_{k} +\frac{1}{\mathbf{n}_{n,r}} - \mathbf{r}_{k+1}\big| \,\leq \,\frac{ C  }{ n^2(1-\mathbf{r}_k)  } $

\item $\Big| \frac{2}{\pi}\Delta_k' \cos^2\big(\frac{\pi}{2}\mathbf{r}_{k+1}   \big) \,    +\,\frac{\eta}{ n }\Big(\log\big( \cos\big(\frac{\pi}{2}\mathbf{r}_{k+1}\big) \big)\,-\,\log\big( \cos\big(\frac{\pi}{2}\mathbf{r}_{k}\big)  \big)    \Big)     \Big| \,\leq \,\frac{ C  }{ n^3(1-\mathbf{r}_k)^2  }$

\item $\left| \frac{2}{\pi}\Delta_k'' \cos^2\big(\frac{\pi}{2}\mathbf{r}_{k+1}   \big)  \,   + \,\frac{ 2 \eta}{n  }\Big( \sin^2\big(\frac{\pi}{2}\mathbf{r}_{k+1}\big) \,-\, \sin^2\big(\frac{\pi}{2}\mathbf{r}_{k}\big)      \Big)   \right| \,\leq \,\frac{ C  }{ n^3(1-\mathbf{r}_k)^2  }$

\item $\Big|\Delta_k''' \cos^2\big(\frac{\pi}{2}\mathbf{r}_{k+1}   \big)  \Big| \,\leq \,\frac{ C  }{ n^3(1-\mathbf{r}_k)^2  }$

\item $\Big|\Delta_k^2\sin\big( \frac{\pi}{2}\mathbf{r}_{k}^{**}  \big)\cos^3\big( \frac{\pi}{2}\mathbf{r}_{k}^{**} \big) \Big| \,\leq \,\frac{ C  }{ n^3(1-\mathbf{r}_k)^2  }$

\end{enumerate}
The inequalities  (iv)--(vii) help us approximate the difference $\mathbf{r}_{k} +\frac{1}{\mathbf{n}_{n,r}} - \mathbf{r}_{k+1}$ within~(\ref{PrePreTib}) below.  More refined inequalities are possible for (v) and (vii). However,  bounds by a multiple of $\frac{ 1 }{ n^3(1-\mathbf{r}_k)^2  }$ are adequate for our purpose; see the observation~(\ref{4Sum}), which is applied in part (F).  The bound (i) follows from~(\ref{ErrorTerm}), the asymptotic $\mathbf{n}_{n,r}\sim n$ for $n\gg 1 $, and the estimates below for $0\leq 1-x\ll 1$:
\begin{align*}
\cos\Big(\frac{\pi}{2}x   \Big) \,=\, \frac{\pi}{2}(1-x)\,+ \,\mathit{O}\big((1-x)^3\big)  \hspace{1cm}\text{and}\hspace{1cm}   \sin\Big(\frac{\pi}{2}x   \Big)\, =\, 1\,+ \,\mathit{O}\big((1-x)^2\big)\,  .
\end{align*}

To get (ii), we  bound $|\Delta_k'|$, 
$|\Delta_k''|$, and $|\Delta_k'''| $ individually.  The bound in (i) for $\Delta_k''' $ is stronger than needed, and we can bound $|\Delta_k''|$ by a multiple of $\frac{1}{n^2(1-\mathbf{r}_k)}$ since $\mathbf{n}_{n,r}\sim n$ when $n\gg 1$ and $\tan(x)$ is bounded by a multiple of $\frac{1}{\frac{\pi}{2}-x}$ for $x\in [0,\frac{\pi}{2})$.  For $|\Delta_k'|$, we observe that the inequality  $\mathbf{r}_k^*-\mathbf{r}_k < \frac{1}{\mathbf{n}_{n,r}} $ and~(\ref{NDef}) imply that 
\begin{align*}
\tan\Big( \frac{\pi}{2}\mathbf{r}_k^*   \Big) \sec^2\Big( \frac{\pi}{2}\mathbf{r}_k^*   \Big)\,=\,\tan\Big( \frac{\pi}{2}\mathbf{r}_k   \Big) \sec^2\Big( \frac{\pi}{2}\mathbf{r}_k   \Big)\,+\, \mathit{O}\bigg(\frac{1}{n(1-\mathbf{r}_k)^4   }\bigg)\,. 
\end{align*} 
By using $ \mathit{O}\big(\frac{1}{n(1-\mathbf{r}_k)^4   }\big)$ in the  above, we  mean that the difference between the terms $\tan\big( \frac{\pi}{2}\mathbf{r}_k^*   \big) \sec^2\big( \frac{\pi}{2}\mathbf{r}_k^*   \big)$ and $\tan\big( \frac{\pi}{2}\mathbf{r}_k   \big) \sec^2\big( \frac{\pi}{2}\mathbf{r}_k   \big)$ is bounded by a constant multiple of $\frac{1}{n(1-\mathbf{r}_k)^4   }$ for all large $n$ and all $k\in \mathbb{N}_0$ with $1-\mathbf{r}_k\geq \frac{\log n  }{2n }$.

The inequality (iii) follows from (ii) and~(\ref{Tintin}). In particular, the factor  $\cos^2\big(\frac{\pi}{2}\mathbf{r}_{k+1}   \big) $ in (iv)--(vii) has the form  $\cos^2\big(\frac{\pi}{2}\mathbf{r}_{k}   \big)+\mathit{O}\big( \frac{1-\mathbf{r}_k  }{n } \big)  $. The bounds (iv)--(vii) follow from   basic calculus estimates.\vspace{.3cm}

\noindent \textit{(E) A consequence of (iii):}  Before going to 
the estimates in part (F) below, we will point out an easy consequence of the bound (iii) in part (D).  If $\big(\ell(n)\big)_{n\in\N}$ is a sequence in $\mathbb{N}_0$ satisfying $1-\mathbf{r}_{\ell(n)}\geq \frac{\log n}{2n}> \frac{1}{\mathbf{n}_{n,r}} $, then the spacing between  the terms in the sequence $(\mathbf{r}_{k})_{0\leq k \leq \ell(n)}$ has the large-$n$ form
$$  \mathbf{r}_{k+1}\,-\, \mathbf{r}_{k}\,=\,\frac{1}{\mathbf{n}_{n,r}} \,+\, \mathit{O}\bigg(\frac{1}{n\log n}  \bigg)\,=\,\frac{1}{n} \,+\, \mathit{O}\bigg(\frac{1}{n\log n}  \bigg)\,,$$
where the errors $\mathit{O}\big(\frac{1}{n\log n}\big)$ are uniformly bounded by a multiple of $\frac{1}{n\log n}$ for all $n\gg 1$ and all $0\leq k<\ell(n)$.
 The second equality above holds since $\mathbf{n}_{n,r}=n+\mathit{O}\big(\log n\big)$. A Riemann sum approximation thus gives us
\begin{align}\label{4Sum}
\sum_{0 \leq k < \ell(n)}\frac{ 1  }{ n^3(1-\mathbf{r}_k)^2  }\,=\,\frac{1+\mathit{o}(1)}{n^2}\int_0^{\mathbf{r}_{\ell(n)}}\frac{1}{(1-x)^2}dx\,=\,\frac{1+\mathit{o}(1)}{n^2} \bigg( \frac{1}{1-\mathbf{r}_{\ell(n)}}  -1\bigg)\,=\,\mathit{o}\Big( \frac{1}{ n } \Big)\,.
\end{align}

\vspace{.2cm}

\noindent \textit{(F) Applying the bounds in (D) to a key telescoping sum:} Assume that $\big(\ell(n)\big)_{n\in\N}$ is a sequence in $\mathbb{N}_0$ satisfying $\ell(n)\leq n$ and   $1-\mathbf{r}_{\ell(n)}\geq \frac{\log n}{2n}> \frac{1}{ \mathbf{n}_{n,r}} $ for all $n$.  Then~(\ref{4Sum}) and the inequalities in part (D) are applicable.  Since $\mathbf{r}_0=0$, the first equality below results from a telescoping sum.
\begin{align}
1\,-\,\mathbf{r}_{\ell(n)}\,=\,&\,\bigg(1\,-\,\frac{\ell(n)    }{ \mathbf{n}_{n,r}  }\bigg)\,+\,\sum_{0 \leq k < \ell(n)}\bigg(\mathbf{r}_{k}+\frac{1}{\mathbf{n}_{n,r}}-\mathbf{r}_{k+1}   \bigg)\nonumber \\
\,=\,&\,\bigg(1\,-\,\frac{\ell(n)    }{ \mathbf{n}_{n,r}  }\bigg)\,+\,\frac{2}{\pi}\sum_{0 \leq k < \ell(n)}\Delta_{k} \cos^2\Big( \frac{\pi}{2}\mathbf{r}_{k+1}  \Big)\,-\,\frac{2}{\pi}\sum_{0 \leq k < \ell(n)}\Delta_{k}^2 \sin\Big( \frac{\pi}{2}\mathbf{r}_{k}^{**}  \Big)\cos^3\Big( \frac{\pi}{2}\mathbf{r}_{k}^{**}  \Big) \nonumber  \\
\,=\,&\,\bigg(1\,-\,\frac{\ell(n)    }{ \mathbf{n}_{n,r}  }\bigg)\,+\,\frac{2}{\pi}\sum_{0 \leq k < \ell(n)}\Delta_{k}' \cos^2\Big( \frac{\pi}{2}\mathbf{r}_{k+1}  \Big)\,+\,\frac{2}{\pi}\sum_{0 \leq k < \ell(n)}\Delta_{k}'' \cos^2\Big( \frac{\pi}{2}\mathbf{r}_{k+1}  \Big)\,+\,\mathit{o}\Big(\frac{1}{n}\Big)   \label{PrePreTib}
\end{align}
The second equality uses the identity~(\ref{Tintin}) to rewrite the difference between $\mathbf{r}_{k}+\frac{1}{\mathbf{n}_{n,r}}$ and $\mathbf{r}_{k+1} $.  In the third equality, we substituted $\Delta_{k}=\Delta_{k}'+\Delta_{k}''+\Delta_{k}'''$ and applied the  bounds in (vi) and (vii) of part (D) along with the observation~(\ref{4Sum}). Furthermore, applying (iv) and (v) of part (D) with~(\ref{4Sum}) again yields that 
\begin{align}
1\,-\,\mathbf{r}_{\ell(n)}\,=\,&\,\bigg(1\,-\,\frac{\ell(n)    }{ \mathbf{n}_{n,r}  }\bigg)\,-\,\frac{\eta}{n}\sum_{0 \leq k < \ell(n)}\,\Big[\log\Big( \cos\Big(\frac{\pi}{2}\mathbf{r}_{k+1}\Big) \Big)\,-\,\log\Big( \cos\Big(\frac{\pi}{2}\mathbf{r}_{k}\Big)  \Big)    \Big]   \nonumber   \\  &\,-\,\frac{2\eta}{n  }\sum_{0 \leq k < \ell(n)}\Big[ \sin^2\Big(\frac{\pi}{2}\mathbf{r}_{k+1}\Big) \,-\, \sin^2\Big(\frac{\pi}{2}\mathbf{r}_{k}\Big)      \Big]  \,+\,\mathit{o}\Big(  \frac{ 1  }{ n} \Big) \nonumber \\
\,=\,&\,\bigg(1\,-\,\frac{\ell(n)    }{ \mathbf{n}_{n,r}  }\bigg)\,-\,\frac{\eta}{ n }\log\Big(\cos\Big(\frac{\pi}{2}\mathbf{r}_{\ell(n)}\Big)\Big)\,-\,\frac{  2\eta}{n  } \sin^2\Big(\frac{\pi}{2} \mathbf{r}_{\ell(n)}\Big)   \,+\,\mathit{o}\Big(  \frac{ 1  }{ n} \Big) \,, \label{PreTib}
\end{align}
 the second equality resulting from telescoping sums and that $\mathbf{r}_0=0$.

 Recall that $\varsigma:= (\log\frac{\pi}{2}+2)\eta  $. By adding and subtracting terms, we can rewrite the equality~(\ref{PreTib}) in the form
\begin{align}
1\,-\,\mathbf{r}_{\ell(n)}\, =\,& \frac{ n-\lfloor \log n\rfloor -\ell(n)   }{ \mathbf{n}_{n,r}  }\,+\,\underbracket{\frac{\lfloor \log n\rfloor + \mathbf{n}_{n,r} -n}{ \mathbf{n}_{n,r}  }\,+\,\frac{\eta\log n-\eta\log\log n}{n} \,-\, \frac{  \varsigma}{n  }}  \,+\, \mathcal{R}_{\ell(n)}^{(n)} \,, \label{Tibb}
\end{align}
where for the error $\mathit{o}\big(\frac{1}{n}\big)$ from~(\ref{PreTib})  we define
\begin{align}\label{RemTerm}
\mathcal{R}_{k}^{(n)}\,:=\,\frac{\eta}{n}\log\bigg(\frac{\frac{\pi}{2}\log n}{n\cos\big(\frac{\pi}{2}r_{k}\big)}   \bigg) \,-\,\frac{  2\eta}{n  } \Big(\sin^2\Big(\frac{\pi}{2} r_{ k}\Big)  \,-\,1\Big) \,+\,\mathit{o}\Big(\frac{1}{n}\Big)  \,.
\end{align}
The difference between the bracketed expression in~(\ref{Tibb}) and the bracketed term below is $\mathit{o}\big(\frac{1}{n}\big) $ since $\mathbf{n}_{n,r}=n-\eta\log n-r+\varsigma +\mathit{o}(1)$ for $n\gg 1$
\begin{align}
1\,-\,\mathbf{r}_{\ell(n)}\, =\,\frac{ n-\lfloor \log n\rfloor -\ell(n)    }{ \mathbf{n}_{n,r}  }\,+\,\underbracket{\frac{ \lfloor \log n\rfloor \, -\, \eta\log \log n \,-\,r}{ n  }} \, + \, \mathcal{R}_{\ell(n)}^{(n)} \,, \label{Tib}
\end{align}
in which we have absorbed the error $\mathit{o}\big(\frac{1}{n}\big) $ of the approximation into $\mathcal{R}_{\ell(n)}^{(n)}$.

\vspace{.3cm}

\noindent \textit{(G) How we can make use of~(\ref{Tib}):}  We will temporarily assume that   $1-\mathbf{r}_{n-\lfloor \log n\rfloor}\geq \frac{\log n}{2n}$ holds for sufficiently large $n$ and show that the asymptotics~(\ref{Flattened}) follows.   
If $1-\mathbf{r}_{n-\lfloor \log n\rfloor}\geq \frac{\log n}{2n}$, then the equality~(\ref{Tib}) holds with $\ell(n)= n-\lfloor \log n\rfloor$, which gives us 
\begin{align}\label{OneMinusR}
1\,-\,\mathbf{r}_{n-\lfloor \log n\rfloor}\,=\,
\frac{ \lfloor \log n\rfloor \, -\, \eta\log \log n \,-\,r}{ n  }\,+\,\mathcal{R}_{n-\lfloor \log n\rfloor}^{(n)} \,.
\end{align}
Note that we can establish~(\ref{Flattened}) by showing  that $\mathcal{R}_{n-\lfloor \log n\rfloor}^{(n)}$ is $\mathit{o}\big(\frac{1}{n}
\big)$ when $n\gg 1$.  Define $c:=\sup_{x\in [0,\pi/2)}\frac{ \frac{\pi}{2}-x  }{\cos(x)  }$.  Since $1-\mathbf{r}_{n-\lfloor \log n\rfloor}\geq \frac{\log n}{2n}$, we can get an upper bound for $\mathcal{R}_{n-\lfloor \log n\rfloor}^{(n)}$ 
by
\begin{align*}
\mathcal{R}_{n-\lfloor \log n\rfloor}^{(n)} \,\leq \,\,&\frac{\eta}{n}\log\bigg(\frac{c\log n}{n\big(1-r_{n-\lfloor \log n\rfloor}\big)}   \bigg) \,+\,\frac{  4\eta}{n  }  \,+\,\mathit{o}\Big(\frac{1}{n}\Big) \\ \,\leq \,\,&\frac{\eta}{n}\log (2c) \,+\,\frac{  4\eta}{n  }  \,+\,\mathit{o}\Big(\frac{1}{n}\Big)\,=\,\mathit{O}\Big(\frac{1}{n}\Big)\,.
\end{align*}
Thus, using~(\ref{OneMinusR}) we can bound $1-\mathbf{r}_{n-\lfloor \log n\rfloor}$ from above and below by constant multiples of $\frac{\log n}{n}$ for $n\gg 1$:
\begin{align*}
\frac{\log n}{2n}\,\leq \,1\,-\,\mathbf{r}_{n-\lfloor \log n\rfloor}\,\leq \,
&\frac{ \lfloor \log n\rfloor \, -\, \eta\log \log n }{ n  }\,+\,\mathit{O}\Big(\frac{1}{n}\Big)\,.
\end{align*}
  It follows from~(\ref{RemTerm}) that $\mathcal{R}_{n-\lfloor \log n\rfloor}^{(n)} $ is $\mathit{O}\big(\frac{1}{n}\big)$. We can thus conclude from~(\ref{OneMinusR}) that $1-\mathbf{r}_{n-\lfloor \log n\rfloor}= \frac{\log n}{n}\big(1+\mathit{o}(1)\big)$.  Plugging this asymptotics for $1-\mathbf{r}_{n-\lfloor \log n\rfloor}$ into~(\ref{RemTerm}) once more, we can conclude that  $\mathcal{R}_{n-\lfloor \log n\rfloor}^{(n)}=\mathit{o}\big(\frac{1}{n}\big)$. Hence~(\ref{Flattened}) holds under the assumption that $1-\mathbf{r}_{n-\lfloor \log n\rfloor}\geq \frac{\log n}{2n}$.

\vspace{.3cm}

\noindent \textit{(H) Establishing the validity of~(\ref{Tib}) when $\ell(n)=n-\lfloor \log n\rfloor$:} It remains to show that $1-\mathbf{r}_{n-\lfloor \log n\rfloor}\geq \frac{\log n}{2n}$ holds for large enough $n$. Let $\ell^*(n)$ be the smallest value in $ \mathbb{N}$ such that
\begin{align}\label{Lower}
1-\mathbf{r}_{\ell^*(n)}\,\leq \,\frac{3 \log n }{4n}\,.
\end{align}
Since $1-\mathbf{r}_{\ell^*(n)-1}> \frac{3 \log n }{4n}$ and $\mathbf{r}_{\ell^*(n)}-\mathbf{r}_{\ell^*(n)-1}=\frac{1}{n}+\mathit{o}(\frac{1}{n})$ by (iii) in part (D), we have  the inequality $1-\mathbf{r}_{\ell^*(n)}\geq \frac{ \log n }{2n}$ for large enough $n$. Thus~(\ref{Tib}) will hold with $\ell(n):=\ell^*(n)$ when $n\gg 1$, which gives the equality below:
\begin{align}\label{Ting}
\frac{3 \log n }{4n}\,\geq\,1-\mathbf{r}_{\ell^*(n)}\, =\, \frac{ n-\lfloor \log n\rfloor -\ell^*(n)  }{ \mathbf{n}_{n,r}  }\,+\, \frac{ \lfloor \log n\rfloor \, -\, \eta\log \log n \,-\,r}{ n  }\,+\,\mathcal{R}^{(n)}_{\ell^*(n)}\,.
\end{align}
Applying the inequality $\cos(x)\leq \frac{\pi}{2}-x $ for $x\in [0,\frac{\pi}{2})$ in~(\ref{RemTerm}),  we get that
\begin{align*}
   \mathcal{R}^{(n)}_{\ell^*(n)}\,\geq\, \frac{\eta}{n}\log\bigg(\frac{\log n}{n(1-r_{\ell^*(n)})}   \bigg) \,+\,\mathit{o}\Big(\frac{1}{n}\Big)\,\geq\, \frac{\eta}{n}\log\Big(\frac{4}{3}   \Big) \,+\,\mathit{o}\Big(\frac{1}{n}\Big) \,=\,\mathit{O}\Big(\frac{1}{n}\Big)\,,
\end{align*}
where the second inequality uses~(\ref{Lower}). The above lower bound for $\mathcal{R}^{(n)}_{\ell^*(n)}$ combined with~(\ref{Ting}) yields
\begin{align}\label{Chub}
\frac{3 \log n }{4n}\, \geq \,& \underbracket{\frac{ n-\lfloor \log n\rfloor -\ell^*(n)    }{ \mathbf{n}_{n,r}  }}_{\text{must be}\,\, <\,0\,\,\text{for large $n$}}\,+\, \underbrace{\frac{ \lfloor \log n\rfloor \, -\, \eta\log \log n }{ n  }\,+\,\mathit{O}\Big(\frac{1}{n}\Big)}_{ > \,\,\frac{3 \log n }{4n} \text{ for large $n$}   }\,.
\end{align}
The first expression on the right side of~(\ref{Chub}) must be negative when $n\gg 1$, and  therefore $\ell^*(n)> n-\lfloor \log n\rfloor$. It follows that $1-\mathbf{r}_{n-\lfloor \log n\rfloor}\geq \frac{\log n}{2n}$ holds for large $n$.
\end{proof}

\subsection{Proof of Lemma~\ref{LemCond}} \label{SecLemCond}
Since the random variables $\mathbb{E}\big[W_n^{\omega}( \beta_{n,r} )\,\big|\, \mathcal{F}_{n}^{\lfloor \log n\rfloor} \big]$  and $W_n^{\omega}( \beta_{n,r} )-\mathbb{E}\big[W_n^{\omega}( \beta_{n,r} )\,\big|\, \mathcal{F}_{n}^{\lfloor \log n\rfloor} \big]  $ are uncorrelated, the square of the $L^2$ distance between $W_n^{\omega}( \beta_{n,r} )$ and $\mathbb{E}\big[W_n^{\omega}\big( \beta_{n,r} \big)\,\big|\, \mathcal{F}_{n}^{\lfloor \log n\rfloor}\big]$ is equal to 
\begin{align*}
  \textup{Var}\Big(W_n^{\omega}( \beta_{n,r} )\Big)\,-\,\textup{Var}\Big( \mathbb{E}\Big[W_n^{\omega}( \beta_{n,r} )\,\Big|\, \mathcal{F}_{n}^{\lfloor \log n\rfloor}\Big]\Big) \,=\,&\, M_{n,r}^{n }(0)\,-\,\textup{Var}\left( \mathcal{Q}^{\lfloor \log n\rfloor}\big\{ X_h^{n,r} \big\}_{h\in E_{\lfloor \log n\rfloor }}\right)\\ 
=\, &\, M_{n,r}^{n }(0)\,-\, M^{\lfloor \log n\rfloor }\left(M_{n,r}^{n-\lfloor \log n\rfloor }(0)\right)\,,
  \end{align*}
where the random variables  $X_h^{n,r}$  are independent copies of  $W_{n-\lfloor \log n\rfloor  }^{\omega}( \beta_{n,r} )$.
The equalities above use~(\ref{RecEqVar}), Proposition~\ref{PropReduce}, and the discussion at the beginning of Section~\ref{SecVarQMap}.
It follows that Lemma~\ref{LemCond} is a corollary of the lemma below.

\begin{lemma}\label{LemmaMMaps} The difference between $M_{n,r}^{n }(0)$ and $M^{\lfloor \log n\rfloor }\big(M_{n,r}^{n-\lfloor \log n\rfloor }(0)\big)$ vanishes as $n\uparrow \infty$.\end{lemma}

\begin{remark}\label{RemarkM} Note that  $M^{\lfloor \log n\rfloor }\big(M_{n,r}^{n-\lfloor \log n\rfloor }(0)\big)$  converges to $R(r)$ as $n\uparrow \infty$.  This follows from Proposition~\ref{PropVar} since $M_{n,r}^{n-\lfloor \log n\rfloor }(0) $, which is equal to the variance of $W_{n-\lfloor \log n\rfloor}^{\omega}\big( \beta_{n,r}    \big) $, has  the large-$n$ asymptotics~(\ref{Unflat}) by Lemma~\ref{LemMtilde}.
\end{remark}

In the proof of Lemma~\ref{LemmaMMaps}, we will use Lemma~\ref{LemMaybe} below, which is a  result from~\cite[Lemma 2.2(iv)]{Clark1}.  Notice that applying the chain rule to the $k$-fold composition of $M(x)=\frac{1}{b}\big[(1+x)^b-1   \big]$ yields
\begin{align*}
 \frac{d}{dx}M^{k}(x)\,=\,\prod_{1\leq j \leq k} \Big( 1+M^{j-1}(x)  \Big)^{b-1}\,=\,(k+1)^2D_{k}\big(M^{k}(x) \big)  \,,   
 \end{align*}
where the function  $D_k:[0,\infty)\rightarrow [0,\infty)$ is defined by
\begin{align}\label{DefDk}
D_{k}(y)\,=\, \frac{1}{(k+1)^2}\prod_{1\leq \ell \leq k} \Big(1+M^{-\ell}(y)\Big)^{b-1}\,. 
\end{align}
In the above,  $M^{-\ell}$ denotes the $\ell$-fold composition of the function inverse of the map $M$. The following lemma gives us uniform bounds for the sequence in $k\in \mathbb{N}_0$ of functions $D_{k}$.
\begin{lemma}\label{LemMaybe}  The sequence of functions $(D_{k})_{k\in \mathbb{N}_0} $ converges  uniformly over any bounded subinterval of $[0,\infty)$ to a limit function $D$.  In particular, $F(L):=\sup_{k\in \mathbb{N}_0}\sup_{x\in [0,L]}D_{k}(x)$ is finite for any $L>0$.
\end{lemma}

\begin{proof}[Proof of Lemma~\ref{LemmaMMaps}] Define $A_{n,r}:=M_{n,r}^{n-\lfloor \log n\rfloor }(0)$.  By Remark~\ref{RemarkM}, $M^{\lfloor \log n\rfloor  }(A_{n,r})$ converges to $R(r)$ as $n\uparrow \infty$. For any $\ell \in \{0,\ldots, \lfloor \log n\rfloor\}$, the definition of $A_{n,r}$ implies that
\begin{align*}
M_{n,r}^{\ell+n-\lfloor \log n\rfloor }(0)-M^{\ell }\left(M_{n,r}^{n-\lfloor \log n\rfloor }(0)\right)
\,=\,&\,M_{n,r}^{\ell} (A_{n,r})\,-\,M^{\ell }(A_{n,r})\\
  \,=\,&\,\sum_{1 \leq k \leq \ell}     \bigg( M_{n,r}^{ k }\Big(M^{\ell-k }(A_{n,r})\Big)\,-\,M^{k-1 }_{n,r}\Big(M^{\ell-k+1 }(A_{n,r})\Big)\bigg)\,. \nonumber  
\end{align*}
 By the mean value theorem, there exists a point $y_k $ in the interval $\Big(M^{\ell-k+1 }(A_{n,r}), M_{n,r}\big(M^{\ell-k }(A_{n,r})\big)\Big)$ for each $k\in \{1,\ldots,\ell\}$ such that the above is equal to
\begin{align}
\,\sum_{1 \leq k \leq \ell} \bigg(  M_{n,r}\Big(M^{\ell-k }(A_{n,r})\Big)\,-\,M\Big(M^{\ell-k }(A_{n,r}) \Big) \bigg)  \frac{d}{dx}M_{n,r}^{ k-1 }(x)\Big|_{x=y_{k} }\,. \nonumber 
\end{align}
Since the derivative of $M_{n,r}^{k-1 } $ is increasing and $M_{n,r}(x)\geq M(x)$ for $x\geq 0$, the above is bounded by
\begin{align}
\sum_{1 \leq k \leq  \ell} \underbrace{\bigg(  M_{n,r}\Big(M^{\ell-k }(A_{n,r})\Big)\,-\,M\Big(M^{\ell-k }(A_{n,r}) \Big) \bigg)}_{ \mathbf{(I)}} \underbrace{\frac{d}{dx}M_{n,r}^{ k-1 }(x)\Big|_{x= M_{n,r}\big(M^{\ell-k }(A_{n,r})\big)  }}_{\mathbf{(II)}  } \,. \label{Pebble}
\end{align}
We will return to~(\ref{Pebble}) after obtaining bounds for the terms $\mathbf{(I)}$ and $\mathbf{(II)}$ individually.\vspace{.3cm}

\noindent \textit{Bound for (I):} The difference between the functions $M_{n,r}$ and $M$ has the bound,
\begin{align}\label{ChangeM}
  M_{n,r}(x)\,-\,M(x) \,=\,\frac{1}{b}(1 +x  )^{b}\Big[\big(1+V_{n,r}\big)^{b-1}-1   \Big]\,<\,V_{n,r} (1 +x  )^{b}\,,
\end{align}
where the inequality holds for large enough $n$ since $V_{n,r}$ is vanishing.  Thus, for large  $n$ we have
\begin{align}
 M_{n,r}\Big(M^{\ell-k }(A_{n,r})\Big)\,-\,M\Big(M^{\ell-k }(A_{n,r}) \Big)\,\leq \,& \, V_{n,r}\Big(1 + M^{\ell-k }(A_{n,r})   \Big)^{b}\nonumber  \\
 \,\leq \,& \, V_{n,r}\Big(1 + M^{\lfloor \log n \rfloor }(A_{n,r})   \Big)^{b}\nonumber \\
 \,<\,& \, 2V_{n,r}\big(1+R(r)   \big)^{b} \,.\label{Zip}
\end{align}
The second inequality uses that $\ell\leq \lfloor \log n\rfloor $ and $x\leq M(x)$ for all $x\geq 0$, and the last inequality holds for large enough $n$ since $M^{\lfloor \log n \rfloor }(A_{n,r})  $ converges to $R(r)$ as $n\uparrow \infty$.
 \vspace{.3cm}

\noindent \textit{Bound for (II):} By the chain rule, the derivative of $M_{n,r}^{ k }$ can be written in the form
\begin{align}
\frac{d}{dx}M_{n,r}^{ k }(x)\,=\,\big(1+ V_{n,r} \big)^{k(b-1)}\prod_{0 \leq j \leq k-1} \Big(1+M_{n,r}^{j}(x)\Big)^{b-1}\,.\nonumber 
\end{align}
Since $V_{n,r}=\mathit{O}\big(\frac{1}{n^2}\big)$ and $k\leq \lfloor \log n\rfloor $, the term $\big(1+ V_{n,r} \big)^{k(b-1)}$ is smaller than $2$ for large $n$.  Moreover, writing  $M_{n,r}^{j}=M_{n,r}^{-(k-j)}M_{n,r}^{k} $ and changing the index of the product to $l=k-j$ gives us the following bound when  $n$ is large:
\begin{align}
\frac{d}{dx}M_{n,r}^{ k }(x) \,\leq \,&\,2\prod_{1\leq l\leq k} \left(1+M_{n,r}^{-l}\Big(M_{n,r}^{ k }(x)\Big)\right)^{b-1} \nonumber \\
 \,\leq \,&\,2\prod_{1\leq l\leq k} \Big(1+M^{-l}\Big(M_{n,r}^{ k }(x)\Big)\Big)^{b-1}\nonumber \\ \,= \, &\, 2(k+1)^2 D_k\Big(M_{n,r}^{ k }(x)\Big)\,,\label{Tibit}  
\end{align}
in which the equality uses the definition~(\ref{DefDk}) of the  function $D_k:[0,\infty)\rightarrow [0,\infty)$. 
To see the second inequality above, notice that  $M_{n,r}^{-1}(y)\leq M^{-1}(y)$ holds for all $y\geq 0$ since $M_{n,r}(x)\geq M(x)$ holds for all $x\geq 0$.  

 An application of~(\ref{Tibit}) to the term $\mathbf{(II)}$ gives us
\begin{align}
\frac{d}{dx}M_{n,r}^{ k-1 }(x)\Big|_{x= M_{n,r}\big(M^{\ell-k }(A_{n,r})\big)  } \,\leq \,\,&2k^2 D_{k-1}\Big(M_{n,r}^{ k }\Big(M^{\ell-k }(A_{n,r})\Big)\Big)\nonumber \\
\,\leq \,\,&2k^2 D_{k-1}\Big(M_{n,r}^{  \ell}(A_{n,r})\Big) \,,\label{Zap}
\end{align}
the second inequality again using that  $M_{n,r}(x)\geq M(x)$ for all $x\geq 0$.

\vspace{.3cm}

\noindent \textit{Returning to~(\ref{Pebble}):} The difference between $M_{n,r}^{\ell} (A_{n,r})$ and $M^{\ell}(A_{n,r})$ is bounded by~(\ref{Pebble}), and we can apply~(\ref{Zip}) and~(\ref{Zap}) to bound~(\ref{Pebble}). Thus, for large enough $n$ and all $0\leq \ell \leq \lfloor \log n\rfloor$ we have
\begin{align}
M_{n,r}^{\ell} (A_{n,r})\,-\,M^{\ell}(A_{n,r}) \,\leq \,\,& 4 V_{n,r}\big(1+R(r)   \big)^{b}\sum_{k=1}^{\ell }   k^2 D_{k-1}\Big(M_{n,r}^{  \ell}(A_{n,r})\Big) \nonumber\\
\,\leq \,\,& c\frac{\lfloor \log n\rfloor^3}{n^2} F\Big(M_{n,r}^{  \ell}(A_{n,r}) \Big)\,,\label{Dip}
\end{align}
for $F(L):=\sup_{k\in \mathbb{N}_0}\sup_{x\in [0,L]}D_k(x)$.   Recall that $F$ is finite-valued by  Lemma~\ref{LemMaybe}. Since $\ell\leq \lfloor \log n \rfloor $ and $V_{n,r}\propto \frac{1}{n^2}$ with large $n$, there is a $c>0$ such that the second inequality above holds for all $n\in \mathbb{N}$.

 Let $\ell^*_{n,r}\in \mathbb{N}$ be the minimum of $\ell=\lfloor \log n\rfloor$ and the largest $\ell$ such that $M_{n,r}^{\ell} (A_{n,r})\leq 2M^{\ell }(A_{n,r})$. Applying~(\ref{Dip}) with $\ell=\ell^*_{n,r}  $ yields the following inequality.
\begin{align}
M_{n,r}^{\ell^*_{n,r}} (A_{n,r})\,-\,M^{\ell^*_{n,r}}(A_{n,r})\,\leq\,\,& c\frac{\lfloor \log n\rfloor^3}{n^2} F\Big(M_{n,r}^{ \ell^*_{n,r}}(A_{n,r}) \Big)\nonumber \\
\,\leq \,\,&c\frac{\lfloor \log n\rfloor^3}{n^2} F\big(3R(r)\big) \,=\,\mathit{O}\bigg(\frac{\log^3 n}{n^2}  \bigg) \label{REW}
\end{align}
To see the second inequality above,  note that since $M_{n,r}^{\ell^*_{n,r}} (A_{n,r})\leq 2M^{\ell^*_{n,r} }(A_{n,r})\leq 2M^{\lfloor \log n\rfloor }(A_{n,r})$ and $M^{\lfloor \log n\rfloor }(A_{n,r})$ converges to $R(r)$ as $n\uparrow \infty$, the value $M_{n,r}^{\ell^*_{n,r}} (A_{n,r})$ is smaller than $3R(r)$ for $n\gg 1$.  Hence the inequality~(\ref{REW}) holds for large enough $n$ because  $F$ is nondecreasing.

We can apply~(\ref{REW}) to get the  inequality below.
\begin{align}
M_{n,r}^{\ell^*_{n,r}+1} (A_{n,r})\,=\,M_{n,r}\Big(M_{n,r}^{\ell^*_{n,r}} (A_{n,r})\Big)\,\leq \,\,&M_{n,r}\bigg(M^{\ell^*_{n,r}} (A_{n,r})\,+\,\mathit{O}\bigg(\frac{\log^3 n}{n^2}  \bigg) \bigg) \nonumber
\\  \,=\,\,&M_{n,r}\Big(M^{\ell^*_{n,r}} (A_{n,r})\Big)\,+\,\mathit{O}\bigg( \frac{ \log^3 n}{n^2}  \bigg) \nonumber \\
\,=\,\,& M^{\ell^*_{n,r}+1} (A_{n,r})\,+\,\mathit{O}\bigg( \frac{ \log^3 n}{n^2}  \bigg)
 \label{TheThe}
\end{align}
Note that $M^{\ell^*_{n,r}} (A_{n,r})$ is smaller than $2R(r)$ for large enough $n$ because $M^{\ell^*_{n,r}} (A_{n,r})\leq M^{\lfloor \log n \rfloor} (A_{n,r}) $ and $M^{\lfloor \log n \rfloor} (A_{n,r}) $ converges to $R(r)$ as $n\uparrow \infty$. Thus, the second equality above holds since the derivative of $M_{n,r}$ is uniformly bounded over bounded intervals.
For the third equality,~(\ref{ChangeM}) implies that  replacing   $M_{n,r}$ by $M$ yields a negligible error.

However,  $M^{\ell^*_{n,r}+1} (A_{n,r})\geq A_{n,r}$, and 
\begin{align*}
    A_{n,r}\,=\, M^{n-\lfloor \log n \rfloor} (0) \,=\, \textup{Var}\Big( W_{n-\lfloor \log n\rfloor}^{\omega}(\beta_{n,r}) \Big) \, \geq \, \frac{\kappa^2}{\lfloor\log n \rfloor},
\end{align*}
 where the inequality holds for large $n$ by~(\ref{Unflat}). This combined with~(\ref{TheThe}) precludes the possibility that $M_{n,r}^{\ell^*_{n,r}+1} (A_{n,r})> 2M^{\ell^*_{n,r}+1} (A_{n,r})$ when $n$ is large.  It follows from the definition of $\ell^*_{n,r}$ that $\ell^*_{n,r}:=\lfloor \log n\rfloor$, and thus~(\ref{REW}) implies that the difference between $M_{n,r}^{\lfloor \log n\rfloor} (A_{n,r})$ and $M^{\lfloor \log n\rfloor}(A_{n,r})$ vanishes with large $n$.
\end{proof}

\subsection{Proof of Lemma~\ref{LemmaHM}}\label{SecLemmaHM}

\begin{proof} It suffices to show that the (uncentered) positive integer moments of $W_{n-\lfloor \log n\rfloor}^{\omega}(\beta_{n,r})$ all converge to one as $n\uparrow \infty$.  For $m\in \{2,3,\ldots\}$, $n\in \mathbb{N}$, $r\in \R$, and $k\in \mathbb{N}_0$ define
$$ \mu_{n,r}^{(m)}(k)\,:=\,\mathbb{E}\Big[ \big(W_{k}^{\omega}(\beta_{n,r})\big)^m \Big]    \hspace{1cm}  \text{and} \hspace{1cm} \nu_{n,r}^{(m)}\,:=\,\mathbb{E}\Big[ \big(\textup{exp}\big\{ \beta_{n,r}\omega -\lambda( \beta_{n,r})\big\} \big)^m \Big]  \,.   $$
Note that $\mu_{n,r}^{(m)}(0)=1$ since $W^{\omega}_{0}(\beta_{n,r})=1$ by definition, and $\mu_{n,r}^{(m)}(k),\nu_{n,r}^{(m)}\geq 1$ by Jensen's inequality. We obtain the following recursive equation in $k\in \mathbb{N}$ by evaluating the $m^{th}$ moment of both sides of the distributional equality~(\ref{PartHierSymmII}):
\begin{align}\label{FT}
\mu_{n,r}^{(m)}(k+1)\,=\, \frac{1}{b^{m-1}}  \big(\mu_{n,r}^{(m)}(k)\big)^{b}\big(\nu_{n,r}^{(m)}\big)^{b-1}\,+\, \mathbf{P}_m\Big( \big(\mu_{n,r}^{(\ell)}(k)\big)^{b}\big(\nu_{n,r}^{(\ell)}\big)^{b-1}\,;\,\, \ell\in \{2,\ldots, m-1\}  \Big)\,,
\end{align}
where $\mathbf{P}_m( y_{2},\ldots,  y_{m-1} )$ is a  polynomial with nonnegative coefficients that sum to $1-\frac{1}{b^{m-1}}$. In particular, $\mathbf{P}_m( y_{2},\ldots,  y_{m-1} )=1-\frac{1}{b^{m-1}}$ when evaluated at $( y_{2},\ldots,  y_{m-1} )=(1,\ldots, 1) $.  Moreover, $ 1-\frac{1}{b^{m-1}}$ is a lower bound for the $\mathbf{P}_m$ term in~(\ref{FT}) since $\mu_{n,r}^{(m)}(k),\nu_{n,r}^{(m)}\geq 1$.

We will use induction to prove that  $\max_{0\leq k\leq  n-\lfloor \log n\rfloor } \big|\mu_{n,r}^{(m)}(k)-1\big|$ vanishes as $n\uparrow \infty$ for each $m\in \{2,3,\ldots\}$.  As a consequence of Lemma~\ref{LemMtilde}, the value $\mu_{n,r}^{(2)}\big(n-\lfloor \log n\rfloor\big)$ converges to one as $n\uparrow \infty$. Since $ \big(\mu_{n,r}^{(2)}(k)\big)_{k\in \mathbb{N}_0} $ is an increasing sequence and $\mu_{n,r}^{(2)}(0)=1$, it follows that $ \max_{0\leq k\leq  n-\lfloor \log n\rfloor }\big| \mu_{n,r}^{(2)}(k)-1\big|$ vanishes as $n\uparrow \infty$.  Suppose, for the purpose of a strong induction argument, that 
\begin{align*}
\max_{0\leq k\leq  n-\lfloor \log n\rfloor }\big|\mu_{n,r}^{(\ell)}(k)-1\big| \hspace{.2cm}\stackrel{n \to \infty}{\longrightarrow }\hspace{.2cm} 0
\end{align*} 
for each $\ell\in\{2,\ldots,m\}$.  Note that $\nu_{n,r}^{(\ell)}$ converges to one as $n\uparrow \infty$ for each $\ell\in \mathbb{N}$ since $\beta_{n,r}$ vanishes with large $n$.  Fix some $\epsilon\in (0,1)$.  Since $\mathbf{P}_{m+1}$ is  continuous and $\mathbf{P}_{m+1}(1,\ldots, 1)=1-\frac{1}{b^{m}}$, we can choose  $n\in \mathbb{N}$ large enough such  that 
\begin{align}\label{IndCond}
\max_{0\leq k\leq  n-\lfloor \log n\rfloor }  \mathbf{P}_{m+1}\Big( \big(\mu_{n,r}^{(\ell)}(k )\big)^{b}\big(\nu_{n,r}^{(\ell)}\big)^{b-1}\,;\,\, \ell\in \{2,\ldots, m\}  \Big)\,\leq \,(1+\epsilon)\Big(1-\frac{1}{b^m}  \Big)\,.
\end{align}
Let $k^*_{n,r,\epsilon}$ be the minimum of $k=n-\lfloor \log n\rfloor $ and the smallest $k\in \mathbb{N}$
with
\begin{align} \label{Smallest}
\big(\mu_{n,r}^{(m+1)}(k)\big)^{b-1}\big(\nu_{n,r}^{(m+1)}\big)^{b-1}\,> \,1+\epsilon \,.
\end{align}
By~(\ref{FT}) and the definition of $k^*_{n,r,\epsilon}$,  we have the   recursive inequality in $k\in \{0,\ldots, k^*_{n,r,\epsilon}-1\}$ below.
\begin{align}\label{Recur<}
\mu_{n,r}^{(m+1)}(k+1)\,\leq \, \frac{1+\epsilon}{b^{m}}  \mu_{n,r}^{(m+1)}(k)\,+\, \mathbf{P}_{m+1}\Big( \big(\mu_{n,r}^{(\ell)}(k)\big)^{b}\big(\nu_{n,r}^{(\ell)}\big)^{b-1}\,;\,\, \ell\in \{2,\ldots, m\}  \Big)
\end{align}
Applying~(\ref{Recur<}) $k^*_{n,r,\epsilon}$ times and using that $\mu_{n,r}^{(m+1)}(0) =1$ yields the first inequality below.
\begin{align}
\max_{0\leq k\leq  k^*_{n,r,\epsilon} } &\mu_{n,r}^{(m+1)}(k) \nonumber \,\\ \leq \,&\, \Big( \frac{1+\epsilon}{b^m}  \Big)^{ k^*_{n,r,\epsilon} }\,+\, \sum_{0 \leq k < k^*_{n,r,\epsilon} }\Big( \frac{1+\epsilon}{b^m}  \Big)^{ k^*_{n,r,\epsilon}-1-k }\mathbf{P}_{m+1}\Big( \big(\mu_{n,r}^{(\ell)}(k)\big)^{b}\big(\nu_{n,r}^{(\ell)}\big)^{b-1}\,;\,\, \ell\in \{2,\ldots, m\}  \Big) \nonumber \\
\,\leq \,& \,\frac{   1}{1-  \frac{1+\epsilon}{b^m}   } \underbracket{\max_{0\leq k\leq  n-\lfloor \log n\rfloor }  \mathbf{P}_{m+1}\Big( \big(\mu_{n,r}^{(\ell)}(k)\big)^{b}\big(\nu_{n,r}^{(\ell)}\big)^{b-1}\,;\,\, \ell\in \{2,\ldots, m\}  \Big)}\label{Liz}
\end{align}
Since $\mathbf{P}_{m+1}\Big( \big(\mu_{n,r}^{(\ell)}(k)\big)^{b}\big(\nu_{n,r}^{(\ell)}\big)^{b-1}\,;\,\, \ell\in \{2,\ldots, m\}  \Big)$ is bounded from below by  $ 1-\frac{1}{b^{m}}$, geometric summation gives us the second inequality above. 
The bracketed term converges to $1-\frac{1}{b^{m}}$ as $n\uparrow \infty$ by the same reasoning as for~(\ref{IndCond}).  We will prove that $k^*_{n,r,\epsilon}= n-\lfloor \log n\rfloor$ holds for large enough $n$ by showing that the condition~(\ref{Smallest}) cannot hold for $k\leq n-\lfloor \log n\rfloor$ when $n\gg 1$.  Notice that
\begin{align*}
\max_{0\leq k\leq   k^*_{n,r,\epsilon} } \Big(&\mu_{n,r}^{(m+1)}(k)\Big)^{b-1}\big(\nu_{n,r}^{(m+1)}\big)^{b-1}\\ &\,\leq \,\underbrace{\bigg(\frac{ 1 }{1-  \frac{1+\epsilon}{b^m}   } \max_{0\leq k\leq  n-\lfloor \log n\rfloor }  \mathbf{P}_{m+1}\Big( \big(\mu_{n,r}^{(\ell)}(k)\big)^{b}\big(\nu_{n,r}^{(\ell)}\big)^{b-1}\,;\,\, \ell\in \{2,\ldots, m\}  \Big)\bigg)^{b-1}   \big(\nu_{n,r}^{(m+1)}\big)^{b-1}}_{ \text{\large This expression converges to $\Big(\frac{ 1 -\frac{1}{b^m} }{1-  \frac{1+\epsilon}{b^m}   }\Big)^{b-1}$ as $n\uparrow \infty$.}}\,.
\end{align*}
Moreover, since $m\geq 2$, the following inequality holds for small $\epsilon>0$:
$$  \bigg(\frac{ 1 -\frac{1}{b^m} }{1-  \frac{1+\epsilon}{b^m}   }\bigg)^{b-1}\,=\,\bigg(1+\frac{  \epsilon }{b^m-1-\epsilon   }\bigg)^{b-1}  \, < \, 1+\epsilon\,.$$  Thus $k^*_{n,r,\epsilon}$ does not satisfy~(\ref{Smallest}) when $n$ is large, and therefore $k^*_{n,r,\epsilon}=n-\lfloor \log n\rfloor$ for large $n$.  Going back to~(\ref{Liz}) with $k^*_{n,r,\epsilon}=n-\lfloor \log n\rfloor$ and applying~(\ref{IndCond}), we get
\begin{align*}
\limsup_{n \to \infty}\max_{0\leq k\leq  n-\lfloor \log n\rfloor } \mu_{n,r}^{(m+1)}(k)\,\leq \,\frac{ (1+\epsilon)(1 -\frac{1}{b^m}) }{1-  \frac{1+\epsilon}{b^m}   }\,.
\end{align*}
Since  $\epsilon>0$ is arbitrary and $\mu_{n,r}^{(m+1)}(k)\geq 1$, the sequence $\big(\max_{0\leq k\leq  n-\lfloor \log n\rfloor } \big|\mu_{n,r}^{(m+1)}(k)-1\big|\big)_{n\in \mathbb{N}}$  is vanishing. Therefore, by induction, $\max_{0\leq k\leq  n-\lfloor \log n\rfloor } \big|\mu_{n,r}^{(m)}(k)-1\big|$ converges to zero for each $m\in \{2,3,\ldots\}$, which completes the proof.
\end{proof}

\end{document}